\documentclass[11pt,a4paper,reqno]{amsart}
\usepackage{amsmath,amssymb,amsfonts,epsfig,yhmath}
\usepackage{color}

\newtheorem{theorem}{Theorem}[section]
\newtheorem{lemma}[theorem]{Lemma}
\newtheorem{prop}[theorem]{Proposition}
\newtheorem{definition}[theorem]{Definition}

\newtheorem{remark}[theorem]{Remark}

\def\eps{\varepsilon}
\def\proofof#1{\begin{proof}[Proof (of~#1).]}

\def\XXint#1#2#3{{\setbox0=\hbox{$#1{#2#3}{\int}$} \vcenter{\vspace{-1pt}\hbox{$#2#3$}}\kern-.5\wd0}}
\def\Xint#1{\mathchoice {\XXint\displaystyle\textstyle{#1}}{\XXint\textstyle\scriptstyle{#1}}{\XXint\scriptstyle\scriptscriptstyle{#1}}{\XXint\scriptscriptstyle\scriptscriptstyle{#1}}\!\int}
\def\intmed{\hbox{\ }\Xint{\hbox{\vrule height -0pt width 10pt depth 1pt}}}

\numberwithin{equation}{section}

\def\R{\mathbb R}
\def\N{\mathbb N}

\def\de0#1{\rule[3pt]{#1}{0.4pt} \hspace{-0.1pt} \rule[3.05pt]{0.05pt}{0.4pt} \hspace{-0.1pt} \rule[3.1pt]{0.05pt}{0.4pt} \hspace{-0.1pt} \rule[3.15pt]{0.05pt}{0.4pt} \hspace{-0.1pt} \rule[3.2pt]{0.05pt}{0.4pt} \hspace{-0.1pt} \rule[3.25pt]{0.05pt}{0.4pt} \hspace{-0.1pt} \rule[3.3pt]{0.05pt}{0.4pt} \hspace{-0.1pt} \rule[3.35pt]{0.05pt}{0.4pt} \hspace{-0.1pt} \rule[3.4pt]{0.05pt}{0.4pt} \hspace{-0.1pt} \rule[3.45pt]{0.05pt}{0.4pt} \hspace{-0.1pt} \rule[3.5pt]{0.05pt}{0.4pt} \hspace{-0.1pt} \rule[3.55pt]{0.05pt}{0.4pt} \hspace{-0.1pt} \rule[3.6pt]{0.05pt}{0.4pt} \hspace{-0.1pt} \rule[3.65pt]{0.05pt}{0.4pt} \hspace{-0.1pt} \rule[3.7pt]{0.05pt}{0.4pt} \hspace{-0.1pt} \rule[3.75pt]{0.05pt}{0.4pt} \hspace{-0.1pt} \rule[3.8pt]{0.05pt}{0.4pt} \hspace{-0.1pt} \rule[3.85pt]{0.05pt}{0.4pt} \hspace{-0.1pt} \rule[3.9pt]{0.05pt}{0.4pt} \hspace{-0.1pt} \rule[3.95pt]{0.05pt}{0.4pt} \hspace{-0.1pt} \rule[4.0pt]{0.05pt}{0.4pt} \hspace{-0.1pt} \rule[4.05pt]{0.05pt}{0.4pt} \hspace{-0.1pt} \rule[4.1pt]{0.05pt}{0.4pt} \hspace{-0.1pt} \rule[4.15pt]{0.05pt}{0.4pt} \hspace{-0.1pt} \rule[4.2pt]{0.05pt}{0.4pt} \hspace{-0.1pt} \rule[4.25pt]{0.05pt}{0.4pt} \hspace{-0.1pt} \rule[4.3pt]{0.05pt}{0.4pt} \hspace{-0.1pt} \rule[4.35pt]{0.05pt}{0.4pt} \hspace{-0.1pt} \rule[4.4pt]{0.05pt}{0.4pt} \hspace{-0.1pt} \rule[4.45pt]{0.05pt}{0.4pt} \hspace{-0.1pt} \rule[4.5pt]{0.05pt}{0.4pt} \hspace{-0.1pt} \rule[4.55pt]{0.05pt}{0.4pt} \hspace{-0.1pt} \rule[4.6pt]{0.05pt}{0.4pt} \hspace{-0.1pt} \rule[4.65pt]{0.05pt}{0.4pt} \hspace{-0.1pt} \rule[4.7pt]{0.05pt}{0.4pt} \hspace{-0.1pt} \rule[4.75pt]{0.05pt}{0.4pt} \hspace{-0.1pt} \rule[4.8pt]{0.05pt}{0.4pt} \hspace{-0.1pt} \rule[4.85pt]{0.05pt}{0.4pt} \hspace{-0.1pt} \rule[4.9pt]{0.05pt}{0.4pt} \hspace{-0.1pt} \rule[4.95pt]{0.05pt}{0.4pt} \hspace{-0.1pt} \rule[5.0pt]{0.05pt}{0.4pt} \hspace{-0.1pt} \rule[5.05pt]{0.05pt}{0.4pt} \hspace{-0.1pt} \rule[5.1pt]{0.05pt}{0.4pt} \hspace{-0.1pt} \rule[5.15pt]{0.05pt}{0.4pt} \hspace{-0.1pt} \rule[5.2pt]{0.05pt}{0.4pt} \hspace{-0.1pt} \rule[5.25pt]{0.05pt}{0.4pt} \hspace{-0.1pt} \rule[5.3pt]{0.05pt}{0.4pt} \hspace{-0.1pt} \rule[5.35pt]{0.05pt}{0.4pt} \hspace{-0.1pt} \rule[5.4pt]{0.05pt}{0.4pt} \hspace{-0.1pt} \rule[5.45pt]{0.05pt}{0.4pt} \hspace{-0.1pt} \rule[5.5pt]{0.05pt}{0.4pt} \hspace{-0.1pt} \rule[5.55pt]{0.05pt}{0.4pt} \hspace{-0.1pt} \rule[5.6pt]{0.05pt}{0.4pt} \hspace{-0.1pt} \rule[5.65pt]{0.05pt}{0.4pt} \hspace{-0.1pt} \rule[5.7pt]{0.05pt}{0.4pt} \hspace{-0.1pt} \rule[5.75pt]{0.05pt}{0.4pt} \hspace{-0.1pt} \rule[5.8pt]{0.05pt}{0.4pt} \hspace{-0.1pt} \rule[5.85pt]{0.05pt}{0.4pt} \hspace{-0.1pt} \rule[5.9pt]{0.05pt}{0.4pt} \hspace{-0.1pt} \rule[5.95pt]{0.05pt}{0.4pt} \hspace{-0.1pt} \rule[6.0pt]{0.05pt}{0.4pt}}	

\title[Particle approximation for nonlocal transport equations]{Deterministic particle approximation for nonlocal transport equations with nonlinear mobility}

\author[M.~Di Francesco]{Marco Di Francesco}
\address{Marco Di Francesco - Dipartimento di Ingegneria e Scienze dell’Informazione e Matematica, Universit\`{a} degli Studi dell'Aquila. Via Vetoio 1, 67100, Coppito (L’Aquila), Italy}
\email{marco.difrancesco@univaq.it}

\author[S.~Fagioli]{Simone Fagioli}
\address{Simone Fagioli - Dipartimento di Ingegneria e Scienze dell’Informazione e Matematica, Universit\`{a} degli Studi dell'Aquila. Via Vetoio 1, 67100, Coppito (L’Aquila), Italy}
\email{simone.fagioli@univaq.it}

\author[E.~Radici]{Emanuela Radici}
\address{Emanuela Radici - Dipartimento di Ingegneria e Scienze dell’Informazione e Matematica, Universit\`{a} degli Studi dell'Aquila. Via Vetoio 1, 67100, Coppito (L’Aquila), Italy}
\email{emanuela.radici@univaq.it}

\begin{document}
\begin{abstract}
We construct a deterministic, Lagrangian many-particle approximation to a class of nonlocal transport PDEs with nonlinear mobility arising in many contexts in biology and social sciences. The approximating particle system is a nonlocal version of the follow-the-leader scheme. We rigorously prove that a suitable discrete piece-wise density reconstructed from the particle scheme converges strongly in $L^1_{loc}$ towards the unique entropy solution to the target PDE as the number of particles tends to infinity. The proof is based on uniform BV estimates on the approximating sequence and on the verification of an approximated version of the entropy condition for large number of particles. As part of the proof, we also prove uniqueness of entropy solutions. We also provide a specific example of non-uniqueness of weak solutions and discuss about the interplay of the entropy condition with the steady states. Finally, we produce numerical simulations supporting the need of a concept of entropy solution in order to get a well-posed semigroup in the continuum limit, and showing the behaviour of solutions for large times.
\end{abstract}

\maketitle

\section{Introduction}

A wide range of phenomena in biology and social sciences can be described by the combination of classical (local) - linear or nonlinear - \emph{diffusion} with some \emph{nonlocal transport} effects. Examples can be found in bacterial chemotaxis \cite{keller_segel,painter_hillen}, animal swarming phenomena \cite{okubo,capasso}, pedestrian movements in a dense crowd \cite{hughes}, and more in general in socio-economical sciences \cite{sznajd,aletti}. In a fairly general setting, a set of $N$ individuals $x_1,\ldots,x_N$ located in a sub-region of the Euclidean space $\R^d$ are subject to a drift which is affected by the status of each other individual. In most of the above-mentioned applications, such a ``biased drift'' can be expressed through a set of first order ordinary differential equations
\begin{equation}\label{eq:intro_discrete}
\dot{x}_i(t)=v[(x_1(t),\ldots,x_N(t)],\qquad i=1,\ldots,N,
\end{equation}
in which the velocity law $v$ is known. Having in mind a particle system obeying the laws of classical mechanics or electromagnetism, the set of equations \eqref{eq:intro_discrete} is quite unconventional due to the absence of inertia. On the other hand, this choice is very common in the modelling of socio-biological systems, mainly due to the following three reasons. 
\begin{itemize}
\item Inertial effects are negligible in many socio-biological aggregation phenomena. Even in cases in which the system is appropriate for a fluid-dynamical description, a `thinking fluid' model, with a velocity field already adjusted to equilibrium conditions, is often preferable compared to a second order approach. The typical examples are in traffic flow and pedestrian flow modelling. Moreover, it is well known in the context of cells aggregation modelling that the time of response to the chemoattractant signal is, most of the times, negligible. Finally, inertia is almost irrelevant in many contexts of socio-economical sciences, such as opinion formation dynamics.
\item First-order modelling turns out to simulate real patterns in concrete relevant situations arising in traffic flow, pedestrian motion, and cell-aggregation, and such an achievement is satisfactory in many situations, in applied fields often lacking a unified rigorous modelling approach.
\item In several practical problem such as the behaviour of a crowd in a panic situation, the model can be seen as the outcome of an optimization process performed externally, in which the "best strategy" needed to solve the problem under study (reaching the exit in the shortest possible time, in the crowd example) is transmitted to the individuals in real time (e.g. a set of ``dynamic" evacuation signals in a smart building).
\end{itemize}

Further to the `discrete' approach \eqref{eq:intro_discrete}, these models are often posed in terms of a ``continuum" PDE approach via a continuity equation
\begin{equation}\label{eq:intro_continuum}
\partial_t \rho + \mathrm{div}(\rho v[\rho]) = 0,
\end{equation}
in which $\rho(\cdot,t)$ is a time dependent probability measure expressing the distribution of individuals on a given region at a given time, and in which the continuum velocity map $v=v[\rho]$ is detected as a reasonable ``cross-grained" version of its discrete counterpart in \eqref{eq:intro_discrete}. The modelling of biological movements and socio-economical dynamics are often simulated at the continuum level as the PDE approach is more easy-to-handle in order to analyse the qualitative behaviour of the whole system, in the form e.g. of the emergence of a specific pattern, or the occurrence of concentration phenomena, or the formation of shock waves or travelling waves. In this regard, the descriptive power of the qualitative properties of the solutions in the continuum setting is an argument in favour of the PDE approach \eqref{eq:intro_continuum}. On the other hand, the intrinsic discrete nature of the applied target situations under study would rather suggest an `individual based' description as the most natural one. For this reason, the justification of continuum models \eqref{eq:intro_continuum} as a \emph{many particle limits} of \eqref{eq:intro_discrete} in this context is an essential requirement to validate the use of PDE models.

As briefly mentioned above, the velocity law $v=v[\rho]$ in the PDE approach \eqref{eq:intro_continuum} may include several effects ranging from diffusion effect to external force fields, from nonlinear convection effects to nonlocal interaction terms. We produce here a non-exhaustive list of results available in the literature in which the continuum PDE \eqref{eq:intro_continuum} is obtained as a limit of a system of interacting particles, with a special focus on \emph{deterministic} particle limits, i.e. in which particles move according to a system of ordinary differential equations (i.e. without any stochastic term). The presence of a diffusion operator has several possible counterparts at the discrete level. The literature on this subject involving probabilistic methods is extremely rich and, by now, well established, see e.g. \cite{varadhan1,varadhan2,presutti} only to mention a few. A first attempt (mainly numerical) to a fully deterministic approach to diffusion equations is due to \cite{russo}, see \cite{gosse} for the case of nonlinear diffusion. 

Without diffusion and with only a local dependency $v=v(\rho)$, an extensive literature has been produced based on probabilistic methods (exclusion processes), see e.g. \cite{Ferrari,Ferrari-Nejjar}. A first rigorous result based on fully deterministic ODEs at the microscopic level for a nonlinear conservation law was recently obtained in \cite{DiFra-Rosini}. Nonlocal velocities $v=W\ast \rho$ have been considered as a special case of the theory developed in \cite{carrillo_choi_hauray}, with $W$ a given kernel (possibly singular) using techniques coming from kinetic equations, see \cite{hauray_jabin}. In all the above mentioned results, the particle system is obtained as a discretised version of the Lagrangian formulation of the system.

A slightly more difficult class of problems is the one in which the velocity $v=v[\rho]$ depends \emph{both locally and non-locally} from $\rho$.
Several results about the mathematical well-posedness of such models are available in the literature, which use either classical nonlinear analysis techniques or numerical schemes. In the paper \cite{colombo} a similar model is studied in the context of pedestrian movements, and the existence and uniqueness of entropy solutions is proven. We also mention \cite{defilippis_goatin}, which covers a more general class of problems, and \cite{amadori_shen} covering a similar model in the context of granular media. A quite general result was obtained in \cite{piccoli_rossi} in which the velocity map $\rho\mapsto v[\rho]$ is required to be Lipschitz continuous as a map from the space of probability measures (equipped with some $p$-Wasserstein distance) with values in $C(\R^d)$, and the authors prove convergence of a time-discretised Lagrangian scheme. We also mention \cite{betancourt}, in which a special class of local-nonlocal dependencies has been considered, however in a different numerical framework. We also recall at this stage the related results in \cite{Bu_DiF_Dol,Bu_Dol_Sch} on the overcrowding preventing version of the Keller-Segel system for chemotaxis, in which the existence and uniqueness of entropy solutions is proven. To our knowledge, no papers in the literature provide (so far) a rigorous result of convergence of a deterministic particle system of the form \eqref{eq:intro_discrete} towards a PDE of the form \eqref{eq:intro_continuum} in the case of local-nonlocal dependence $v=v[\rho]$. Indeed, the result in \cite{piccoli_rossi} does not apply to this case in view of the Lipschitz continuity assumption on the velocity field, see also a similar result in \cite{goatin_rossi}.

In this paper we aim at providing, for the first time, a rigorous deterministic many-particle limit for the one-dimensional \emph{nonlocal interaction equation} with \emph{nonlinear mobility}
\begin{equation}\label{eq:intro_PDE}
  \partial_t \rho -\partial_x (\rho v(\rho) K\ast \rho) = 0,
\end{equation}
in which $v$ and $K$ satisfy the following set of assumptions:
\begin{itemize}
\item[(Av)] $v \in C^1([0,+\infty))$ is a decreasing function such that $v(0)=v_{max}>0$, $v(M)=0$ for some $M>0$, $v'<0$ on interval $(0,M]$, $v\equiv 0$ on $[M,+\infty)$.
\item[(AK)] $K \in C^2(\R)$, $K(0)=0$ (without restriction), $K(x)=K(-x)$ for all $x\in \R$, $K'(x)>0$ for $x>0$, $K'' \in \mathrm{Lip}_{loc}(\R)$.
\end{itemize}
Also in view of the applications in mind, the unknown $\rho=\rho(x,t)$ in \eqref{eq:intro_PDE} will be assumed to be non-negative throughout the whole paper. The PDE \eqref{eq:intro_PDE} is coupled with an initial condition
\begin{equation}\label{eq:intro_initial}
  \rho(x,0)=\bar{\rho}(x),\qquad \bar{\rho}\in L^\infty(\R)\cap BV(\R),\,\,\,0\leq \bar{\rho}(x)\leq M,\,\,\,\hbox{$\mathrm{supp}(\bar{\rho})$ compact}.
\end{equation}
The constant $M$ here plays the role of a \emph{maximal density}, which is supposed not to be exceeded by the density for all times. Clearly, the property $\rho\in [0,M]$ has to be proven to be invariant with respect to time. We notice that the total mass of $\rho$ in \eqref{eq:intro_PDE} is formally conserved. For simplicity, throughout the paper we shall set
 We set
\begin{equation*}
 \| \bar{\rho} \|_{L^1(\R)}=1\,.
\end{equation*}
We set $[\bar{x}_{min}, \bar{x}_{max}]$ as the closed convex hull of $\mathrm{supp}\bar{\rho}$.

Our goal is to approximate rigorously the solution $\rho$ to \eqref{eq:intro_PDE} with initial datum $\bar{\rho}$ via a set of moving particles. More precisely, we aim to proving that the \emph{entropy solution} of the Cauchy problem for \eqref{eq:intro_PDE} can be obtained as the large particle limit of a discrete Lagrangian approximation of the form \eqref{eq:intro_discrete}. Such a Lagrangian approximation can be introduced as follows as a reasonable generalization of particle approximations considered previously in the literature in \cite{DiFra-Rosini,DiFra-Fagioli-Rosini,DiFra-Fagioli-Rosini-Russo1,DiFra-Fagioli-Rosini-Russo2}. For a fixed integer $N$ sufficiently large, we split $[\bar{x}_{min}, \bar{x}_{max}]$ into $N$ intervals such that
the integral of the restriction of $\bar{\rho}$ over each interval equals $1/N$. More precisely, we let $\bar{x}_0= \bar{x}_{min}$ and $\bar{x}_N = \bar{x}_{max}$, and define recursively the points $\bar{x}_i$ for $i \in \{ 1,\, \ldots,\, N-1\}$ as
\begin{equation}\label{eq:dscr_IC}
\bar{x}_i = \sup \left\lbrace x \in \R : \int_{\bar{x}_{i-1}}^x \bar{\rho}(x) dx < \frac{1}{N}  \right\rbrace\,.
\end{equation}
It is clear from the construction that $\int_{\bar{x}_{N-1}}^{\bar{x}_N} \bar{\rho}(x) dx = 1/N$ and $\bar{x}_0 < \bar{x}_1 < \ldots\, < \bar{x}_{N-1} < \bar{x}_N$.
Consider then $N+1$ particles located at initial time at the positions $\bar{x}_i$ and let them evolve accordingly to the following system ODEs
\begin{equation}\label{Odes}
\dot{x}_i(t) =   - \frac{v(R_i(t))}{N} \sum_{j > i} K'(x_i(t) - x_j(t)) - \frac{v(R_{i-1}(t))}{N} \sum_{j < i} K'(x_i(t) - x_j(t))\,,
\end{equation}
with $i\in\{0,\ldots,N\}$, where the discrete density $R_i(t)$ is defined as follows
\[ R_i (t):= \frac{1}{N(x_{i+1}(t) - x_i(t))},\qquad i=0,\ldots, N-1. \]
In \eqref{Odes}, each particle $x_i$ has mass $1/N$. We are then in position to define the $N$-discrete density
\begin{equation}\label{eq:discrete_density}
\rho^N(t,\,x):= \sum_{i=0}^{N-1}  R^N_i (t) \chi_{[x_i(t),\,x_{i+1}(t))}(x).
\end{equation}
We observe that $\rho^N(t,\cdot)$ has total mass equal to $1$ for all times. We refer to system \eqref{Odes} as \emph{non-local Follow-the-leader} scheme, as in fact this system is a non-local extension of the classical Follow-the-leader scheme previously considered in the literature. More in detail, system \eqref{Odes} is motivated as follows. The right-hand side of \eqref{Odes} represents the velocity of each particle. Therefore, it has to be reminiscent of a discrete Lagrangian formulation of the Eulerian velocity $-v(\rho)K'\ast \rho$ in the continuity equation \eqref{eq:intro_PDE}. Now, since we are in one-space dimension, the discrete density $R_i$ is a totally reasonable replacement for the continuum density $\rho$, except that one has to decide whether the discrete density should be constructed in a forward, backward, or centred fashion. Our choice of splitting the velocity $\dot{x}_i$ into a backward and forward term is motivated by the sign of the nonlocal interaction $K'(x_i-x_j)$, which is concordant with the sign of $x_i - x_j$. Hence, since $K'(x)$ is negative on $x<0$, particles labelled by $x_j$ with $x_j>x_i$ yield a drift on $x_i$ oriented towards the \emph{positive} direction. Since the role of the nonlinear mobility term $\rho v(\rho)$ is that of preventing overcrowding at high densities (consistently with the assumption of $v$ being monotone decreasing), such a drift term should be ``tempered" by the position of the $(i+1)$-th particle. This motivates the use or $v(R_i)$ in the sum with $x_j>x_i$. A symmetric argument justifies the use of $v(R_{i-1})$ in the remaining part of the sum with $x_j<x_i$.

Our main results concerns with the study of the many particle limit as $N \rightarrow \infty$ for the discrete density $\rho^N$ defined above. Apart from the above mentioned assumptions on $v$ and $K$ and $\bar\rho$, we shall also assume that $\bar\rho \in BV(\R)$. Such a condition is crucial in order to prove the needed estimate which guarantee that $\rho^N$ converge (up to a subsequence) to some limiting density $\rho$ in a strong enough topology. As a minimal requirement, the limit $\rho$ should satisfy \eqref{eq:intro_PDE} in a distributional sense. On the other hand, the presence of a nonlinear convection in \eqref{eq:intro_PDE} suggests the possibility of multiple weak solutions for fixed initial data. A notion of \emph{entropy solution} in the sense of Kruzkov \cite{kruzkov} is therefore needed to secure uniqueness. Motivated by this remark, we shall actually prove that the limit density $\rho$ of the above particle scheme is an entropy solution to \eqref{eq:intro_PDE} with initial condition $\bar\rho$, in the sense of the following definition.

\begin{definition}[Entropy solution]\label{solentropicadef}
Let $\bar\rho\in L^\infty(\R)\cap L^1_+(\R)$. Denoting $f(z):= zv(z)$, we say that $\rho:[0,+\infty)\times \R\rightarrow [0,+\infty)$ is an \emph{entropy solution} of~\eqref{eq:intro_PDE} with initial condition $\bar\rho$ if $\rho \in L^{\infty}([0,\infty), L^1(\R,[0,1]))$ and, for all constants $c\geq 0$ and for all $\varphi \in \mathcal{C}^{\infty}_c ([0,+\infty)\times \R)$ with $\varphi\geq 0$ one has
\[ 0 \leq \int_{\R} |\bar\rho(x) -c|\varphi(0,x)\,dx+\int_0^{+\infty} \int_{\R} |\rho -c|\varphi_t - \mathrm{sign}(\rho - c)[(f(\rho) - f(c))K' \ast \rho\varphi_x - f(c)K'' \ast \rho \varphi]\,dxdt. \]
\end{definition}

We are now ready to state the main result of our paper.

\begin{theorem}\label{main}
Assume $v$ and $K$ satisfy (Av) and (AK) respectively. Let $\bar{\rho} \in BV(\R)\cap L^1_+(\R)$ be a compactly supported function with total unit mass and such that $\bar\rho\leq M$.
Then, for all $T>0$, the discrete density $\rho^N$ constructed in \eqref{eq:discrete_density} converges almost everywhere and in $L^1([0,\,T] \times \R)$ to the unique entropy solution $\rho$ of the Cauchy problem
\begin{equation}\label{CauchyProblem}
\left\lbrace \begin{array}{ll}
\partial_t \rho = \partial_x(\rho v(\rho) K' \ast \rho) &(t,\,x) \in (0,\,T] \times \R\,,\\
\rho(0,\,x) = \bar{\rho}(x) &x \in \R\,.
\end{array}\right.
\end{equation}
\end{theorem}

As a by-product, the above result also imply existence of entropy solutions for \eqref{CauchyProblem}, a task which has been touched in other papers previously \cite{colombo,defilippis_goatin,Bu_DiF_Dol,betancourt}.
Implicitly, our results also asserts the uniqueness of entropy solutions for \eqref{eq:intro_PDE}, a side result that we shall prove as well in the paper, similarly to what done in \cite{Bu_DiF_Dol,Bu_Dol_Sch}.

The need of the entropy condition to define a suitable notion of solution semigroup for \eqref{eq:intro_PDE} is not only motivated by the possibility of proving its uniqueness. We actually prove in the paper that a mere notion of weak solution does not infer the well-posedness of the semigroup as multiple weak solution can be produced with the same initial condition. 

Our paper is structured as follows. In Section \ref{sec:2} we introduce the nonlocal follow-the-leader particle scheme and prove that it satisfies a discrete maximum principle, a crucial ingredient in order to deal with the particle approximation in the sequel of the paper. In Section \ref{sec:convergence} we prove all the estimates needed in order to detect strong $L^1$ compactness for the approximating sequence $\rho^N$. The main ingredient of this section is the $BV$ estimate proven in Proposition \ref{totalvariation}. We emphasize that the presence of an \emph{attractive} interaction potential in the particle system implies most likely a \emph{growth} w.r.t. time of the total variation. Therefore, one has to check that the blow-up in finite time of the total variation is avoided. In Section \ref{sec:consi}, we prove that the limit of the approximating sequence is an entropy solution in the sense of Definition \ref{solentropicadef}. This task is quite technical as it requires checking a discrete version of Kruzkov's entropy condition. In Section \ref{sec:discussion} we provide an explicit example of non uniqueness of weak solutions, which has links with the admissibility of steady states. Finally, in Section \ref{sec:numerics} we complement our results with numerical simulations.

\section{The non-local Follow-the-leader scheme}\label{sec:2}

In this section we introduce and analyse in detail our approximating particle scheme \eqref{Odes}. Here the macroscopic variable $\rho$ does not need to be labelled by $N$, as $N$ is supposed to be fixed throughout the whole section. The regularity assumptions on $v$ and $K$ in (Av) and (AK) imply that the right-hand side of \eqref{Odes} is locally Lipschitz with respect to the $N+1$-tuple $(x_0,x_1,\ldots,x_N)$ as long as we can guarantee that the denominator in $R_i$ does not vanish. Such a property is a consequence of the following  \emph{Discrete Maximum Principle}, ensuring that the particles cannot touch each other at any time. This implies both the (global-in-time) existence of solutions of the system~\eqref{Odes} for all times $t>0$, and the conservation of the initial particle ordering during the evolution.

\begin{lemma}[Discrete Maximum Principle]\label{lem:maximum}
Let $N\in\N$ be fixed and assume that (Av) and (AK) hold. In particular, let $M>0$ be as in assumption (Av). Let $\bar{x}_0<\bar{x}_1<\ldots<\bar{x}_N$ be the initial positions for \eqref{Odes}, and assume that
\begin{equation}\label{eq:MPcondition}
\bar{x}_{i+1}-\bar{x}_i \geq \frac{1}{MN}
\end{equation}
Then every solution $x_i(t)$ to the system~\eqref{Odes} satisfies
\begin{equation}\label{MaxPrinc}
\frac{1}{MN} \leq  x_{i+1}(t) - x_i(t) \qquad \hbox{for all $i \in \{0,\,\ldots,\,N-1\}$ and for all $t \in [0,\,+\infty)$}.
\end{equation}
Consequently, the unique solution $(x_0(t),\ldots,x_N(t))$ to \eqref{Odes} with initial condition $(\bar{x}_0,\ldots,\bar{x}_N)$ exists globally in time.
\end{lemma}

\begin{proof}
Let $T_{max}>0$ be the maximal existence time for \eqref{Odes}. Due to the assumptions (Av) and (AK), the local-in-time solution $(x_0(t),\ldots,x_N(t))$ is $C^1$ on $[0,T_{max})$. If we prove that \eqref{MaxPrinc} holds on $[0,T_{max})$, this will automatically prove global existence by a simple continuation principle. Arguing by contradiction, assume that $t_1< T_{max}$ is the first instant where two consecutive particles are the distance $1/MN$ and get closer afterwards, i.e.
\[t_1=\inf\{ t \in [0,\,T] : \,\,\hbox{there exists}\, i\, : x_{i+1}(t) - x_i(t) = 1/MN \},\]
and there exists $t_2 \in (t_1,\,T]$ such that
\[x_{i+1}(t) - x_i(t) < \frac{1}{MN} \qquad \forall t \in (t_1,\, t_2]\,.   \]
Notice that the minimality of $t_1$ ensures that all particles maintain their initial order for all $t \in [0,\,t_1)$. At time $t_1$ we have $R_i(t_1)=0$ due to (Av). Substituting this value in the equation \eqref{Odes} for $x_i$, we easily see that only the terms $j<i$ survive in the nonlocal part, thus yielding $\dot{x}_i(t_1) \leq 0$. Similarly, we get $\dot{x}_{i+1}(t_1) \geq 0$. For similar reasons, if $\dot{x}_{i+1}(t_1) = 0$ then the ODE for $x_{i+1}$ implies that at time $t_i$ we have $R_{i+1}(t_1)=M$, or equivalently $x_{i+2}(t_1) - x_{i+1}(t_1)= 1/MN$. Similarly, if $\dot{x}_{i}(t_1)=0$ then $x_{i}(t_1) - x_{i-1}(t_1)= 1/MN$. Let us now assume for the moment that $x_{i+2}(t_1) - x_{i+1}(t_1)= x_{i}(t_1) - x_{i-1}(t_1) = 1/MN$. Then, with similar arguments as above one can show that $\dot{x}_{i-1}(t_1)\leq 0$ and $\dot{x}_{i+2}(t_1)\geq 0$, and we can repeat the same argument above to obtain that $\dot{x}_{i-1}(t_1)= 0$ implies $x_{i-1}(t_1) - x_{i-2}(t_1) = 1/MN$ and $\dot{x}_{i+2}(t_1)=0$ implies $x_{i+3}(t_1) - x_{i+2}(t_1)= 1/MN$. Such a procedure can be iterated to conclude that there exists either some index $k\geq i$ with $\dot{x}_{k+1}(t_1)>0$ or some index $h\leq i$ such that $\dot{x}_k(t_1)<0$, otherwise any two consecutive particles would be placed at distance $1/MN$ and the system would be static for all $t \in (t_1,\,T]$, which would contradict the existence of $t_2$.

The above considerations imply that we can assume, without loss of generality, that
\[ \dot{x}_{i+1}(t_1) >0,\quad \mbox{  and  } \quad\dot{x}_i(t_1) \leq 0\,.  \]
Let $\eps_{i+1}>0$ be small enough such that $t_1+ \eps_{i+1} < t_2$, then by Taylor expansion one has
\[  x_{i+1}(t) = x_{i+1}(t_1) +  \dot{x}_{i+1}(t_1)(t-t_1) + o(|t-t_1|)\,,  \]
where, up to taking $\eps_{i+1}$ even smaller, the contribute $o(t-t_1)$ does not affect the sign of $\dot{x}_{i+1}(t_1)(t-t_1)$. As a consequence, $x_{i+1}(t) > x_{i+1}(t_1)$ for all $t\in (t_1,t_1 +\eps_{i+1})$ and a symmetric argument gives also $x_i(t) \leq x_i(t_1)$ for all $t\in (t_1,t_1+\eps_{i})$. In particular, we deduce that
\[ x_{i+1}(t) - x_i(t) \geq  x_{i+1}(t_1) - x_i(t_1) = \frac{1}{MN} \quad \forall t \in (t_1,\, t_1 + \min\{\eps_{i},\,\eps_{i+1} \}) \]
and this contradicts the existence of $t_2$. This argument ensures both the validity of~\eqref{MaxPrinc} and the existence of solutions for all times $t>0$.
\end{proof}

Let us consider the discrete density
\begin{equation*}
\rho(t,\,x):= \sum_{i=0}^{N-1}  R_i (t) \chi_{[x_i(t),\,x_{i+1}(t))}(x).
\end{equation*}
A straightforward consequence of Lemma \ref{lem:maximum} is that
\[\rho(t,x)\leq M\qquad \hbox{for all $(t,x)\in [0,+\infty)\times\R$}.\]
Moreover, we observe that $\rho$ has unit mass on $\R$ for all times.

As already mentioned before, a straightforward consequence of the above Maximum Principle is that the particles can never touch or cross each other. In particular, the particle $x_0$ will have no particles at its left for all times, which means that the ODE for $x_0$ will only feature terms with $j>0$ on the nonlocal sum. A symmetric statement holds for $x_N$. As a consequence of that $\dot{x}_0(t) \geq 0$ and $\dot{x}_N(t) \leq 0$ for all $t$, thus the support of $\rho^N(t,\,\cdot)$ is bounded by $\ell$ uniformly in $N$ and $t$.
We summarize this property in the next lemma.

\begin{lemma}\label{lem:support}
Under the same assumptions of Lemma \ref{lem:maximum}, the support of $\rho(t,\cdot)$ is contained in the interval $[\bar{x}_0,\bar{x}_N]$ for all times $t\in [0,+\infty)$.
\end{lemma}

\section{Convergence of particle scheme}\label{sec:convergence}

We now focus on the converge of the particle scheme \eqref{Odes}, where the initial condition \eqref{eq:dscr_IC} is constructed from an $L^\infty(\R)$ initial density $\bar\rho$ having compact support and finite total variation.

The proof of Theorem~\ref{main} relies on two main steps: the first one consists in proving that the discrete density $(\rho^N)$ defined in \eqref{eq:discrete_density} is strongly convergent (up to a subsequence) to a limit $\rho$ in $L^1([0,T] \times \R)$, the second one is to show that the limit $\rho$ is a weak entropy solution of~\eqref{CauchyProblem} according to Definition \ref{solentropicadef}. In this section we take care of the former step. As we will show in Propositions~\ref{totalvariation} and~\ref{continuitytime} below, the sequence $(\rho_N)_{N\in \N}$ satisfies good compactness properties with respect to the space variables but, on the other hand, we cannot reach a uniform $L^1$ control on the time oscillations. In our case, we are only able to prove a uniform time continuity estimate with respect to the $1$-Wasserstein distance (see \cite{villani_book}), which nevertheless will suffice to achieve the required compactness in the product space. Such a strategy recalls the one used in \cite{DiFra-Rosini} for the case of a scalar conservation law. The main result of this section is the content of the following

\begin{theorem}\label{convergence}
Under the assumptions of Theorem~\ref{main}, the sequence $\rho^N$ is strongly relatively compact in $L^1([0,T] \times \R)$
\end{theorem}

The proof of Theorem~\ref{convergence} relies on a generalized statement of the celebrated Aubin-Lions Lemma (see \cite{RS,DiFra-Matthes,DiFra-Fagioli-Rosini}) that we recall here for the reader's convenience. In what follows, $d_1$ is the $1$-Wasserstein distance.

\begin{theorem}[Generalized Aubin-Lions Lemma]\label{thm:aubin}
Let $\tau>0$ be fixed. Let $\eta^N$ be a sequence in $L^{\infty}((0,\,\tau); L^1(\R))$ such that
$\eta^N(t,\,\cdot) \geq 0$ and $\| \eta^N(t,\,\cdot) \|_{L^1(\R)}=1$ for every $N\in\N$ and $t\in [0,\,\tau]$.
If the following conditions hold
\begin{enumerate}
\item[I)] $\sup_{N}  \int_0^{\tau} \left[\|\eta^N(t,\,\cdot)\|_{L^1(\R)}dt + TV\big[ \eta^N(t,\,\cdot)\big]+ \mathrm{meas}(\mathrm{supp}[\eta^N(t,\cdot)])\right]dt < \infty$,
\item[II)] there exists a constant $C>0$ independent from $N$ such that $d_{W^1}\big( \eta^N(t,\,\cdot), \eta^N(s,\,\cdot) \big) < C |t-s|$ for all $s,\,t \in (0,\,\tau)$,
\end{enumerate}
then $\eta^N$ is strongly relatively compact in $L^1([0,\,\tau]\times \R)$.
\end{theorem}

In view of Theorem \ref{thm:aubin}, the result in Theorem \ref{convergence} will follow as a consequence of the following two propositions.

\begin{prop}\label{totalvariation}
Let $\bar{\rho},\,v,\,K$ and $T$ be as in the statement of Theorem~\ref{main}. Then, there exists a positive constant $C>0$ (only depending on $K$, $v$, and on $\mathrm{supp}(\bar\rho)$) such that for every $N \in \N$ one has
\begin{equation}\label{TV}
TV[\rho^N(t,\,\cdot)] \leq TV[\bar{\rho}] e^{Ct} \qquad \hbox{for all $t \in [0,T]$}\,.
\end{equation}
\end{prop}

\begin{prop}\label{continuitytime}
Let $\bar{\rho},\,v,\,K$ and $T$ be as in the statement of Theorem~\ref{main}. Then, there exists a positive constant $C>0$ (only depending on $K$) such that
\begin{equation}\label{contime}
d_{W^1}\big( \rho^N(t,\,\cdot), \rho^N(s,\,\cdot) \big) < C |t-s| \quad \hbox{for all $s,\,t \in (0,\,T)$, and for all $N \in \N$}\,.
\end{equation}
\end{prop}

The remaining part of this section is devoted to prove Propositions~\ref{totalvariation} and~\ref{continuitytime}.
For future use we compute
\begin{align}
\dot{R}_i(t) =& -N(R_i)^2 (\dot{x}_{i+1} - \dot{x}_i) = -N(R_i)^2 \Big[ -2v(R_i)\frac{1}{N}K'(x_{i+1} - x_i) \nonumber  \\
& -(v(R_{i+1}) - v(R_i))\frac{1}{N}\sum_{j > i+1} K'(x_{i+1} - x_j)\nonumber\\
& - v(R_i) \frac{1}{N}\sum_{j > i+1}\big( K'(x_{i+1} - x_j) - K'(x_i - x_j) \big) \nonumber\\
&  -(v(R_i) - v(R_{i-1}))\frac{1}{N}\sum_{j<i} K'(x_i - x_j) - v(R_i)\frac{1}{N}\sum_{j<i} \big(K'(x_{i+1} - x_j) - K'(x_i - x_j) \big) \Big]\,.\label{eq:explcit_Rdot}
\end{align}

\proofof{Proposition~\ref{totalvariation}}
It is easy to see that $TV[\rho^N(0,\,\cdot)] \leq TV[\bar{\rho}]$. Then estimate~\eqref{TV} follows by Gronwall Lemma as soon as we show that
\begin{equation}\label{dTVlimitata}
 \frac{d}{dt} TV[\rho^N(t,\,\cdot)] \leq C \,TV[\rho^N(t,\,\cdot)],
\end{equation}
for a suitable constant $C>0$. The total variation of $\rho^N$ at time $t$ is given by
\begin{align*}
T&V[\rho^N(t,\,\cdot)] = R_0(t) + R_{N}(t) + \sum_{i=0}^{N-1} |R_{i+1}(t) - R_i(t)| \\
&=\sum_{i=1}^{N-1}R_i [\mathrm{sign}(R_i - R_{i-1}) - \mathrm{sign}(R_{i+1}-R_i)] - R_0 (\mathrm{sign}(R_1-R_0) -1)\\
&\, + R_N( \mathrm{sign}(R_N - R_{N-1}) + 1) \\
&=\mu_0(t)R_0(t)+\mu_N(t)R_N +  \sum_{i=1}^{N-1}R_i \mu_i,
\end{align*}
where we set for brevity
\begin{align*}
 & \mu_i(t):=\mathrm{sign}(R_i(t) - R_{i+1}(t)) - \mathrm{sign}(R_{i-1}(t) - R_i(t))\qquad i=1,\ldots,N-1,\\
 & \mu_0(t)= \big( 1- \mathrm{sign}(R_1 -R_0) \big),\\
 & \mu_N(t)=\big(1+ \mathrm{sign}(R_N -R_{N-1}) \big).
\end{align*}
Then we can compute
\begin{align*}
 \frac{d}{dt} TV[\rho^N(t,\,\cdot)] &= \dot{R}_0(t) +\dot{R}_{N}(t) + \sum_{i=0}^{N-1} \mathrm{sign}\big(R_{i+1}(t) - R_i(t)\big)\big( \dot{R}_{i+1}(t) - \dot{R}_i(t) \big) \\
&= \mu_0(t)\dot{R}_0(t) + \mu_N(t) \dot{R}_N(t) + \sum_{i=1}^{N-1} \mu(R_i(t))\dot{R}_i(t)\,.
\end{align*}
The value of the coefficient $\mu_i(t)$ clearly depends on the positions of the consecutive particles, it is easy to see that for $i \in \{ 1,\,\ldots,\,N-1\}$
\begin{equation*}
\mu_i(t)= \left\lbrace\begin{array}{lll}
-2 \quad &\mbox{if $R_{i+1} > R_i$ and $R_{i-1}> R_i$},\\
2 \quad  &\mbox{if $R_{i+1} < R_i$ and $R_{i-1}< R_i$},\\
0 \quad  &\mbox{if $R_{i+1} \geq R_i \geq R_{i-1}$ or $R_{i-1}\geq R_i \geq R_{i+1}$,}
\end{array}
\right.
\end{equation*}
moreover
\begin{equation*}
\mu_0(t)=\left\lbrace \begin{array}{ll}
0 \quad \mbox{if $R_1 < R_0$,}\\
2  \quad \mbox{if $R_1 > R_0$,}
\end{array} \right.
\qquad
\mu_N(t)= \left\lbrace\begin{array}{ll}
0 \quad \mbox{if $R_{N-1} > R_N$,}\\
2  \quad \mbox{if $R_{N-1} < R_N$.}
\end{array}\right.
\end{equation*}
Recalling \eqref{eq:explcit_Rdot}, we can rewrite
\begin{equation}\label{eq:ddt_estimate}
\frac{d}{dt} TV[\rho^N(t,\,\cdot)] = \mu_0(t)\dot{R}_0(t) + \mu_N(t) \dot{R}_N(t) - \sum_{i=1}^{N-1} \mu_i(t)(R_i(t))^2 \emph{I}_i -  \sum_{i=1}^{N-1} \mu_i(t)R_i(t) \emph{II}_i\,,
\end{equation}
where
\[ \emph{I}_i = -\big(v(R_{i+1}(t)) - v(R_i(t))\big)\sum_{j > i+1} K'(x_{i+1}(t) - x_j(t)) - \big(v(R_i(t)) - v(R_{i-1}(t))\big)\sum_{j < i} K'(x_i(t) - x_j(t))\,, \]
and
\begin{align*}
 &  \emph{II}_i = -R_i(t)v(R_i(t)) \sum_{j \neq i,\,i+1} \big( K'(x_{i+1}(t) - x_j(t)) - K'(x_i(t) - x_j(t)) \big)\\
 & \,\,  - 2R_i(t)v(R_i(t))K'(x_{i+1}(t)-x_i(t))\,.
\end{align*}
Let us first estimate $-\sum_{i=1}^{N-1}\mu_i(t)(R_i(t))^2\emph{I}_i$ in \eqref{eq:ddt_estimate}. Clearly, the only relevant contributions in the sum come from the particles $x_i$ for which $\mu_i(t)\neq 0$. However, if the index $i$ is such that $\mu_i(t)=-2$, then $R_{i+1}, R_{i-1} > R_i$ and the monotonicity of $v$ implies
\[ v(R_{i+1}(t)) - v(R_i(t))  <0\,, \quad\mbox{ and }\quad v(R_i(t))-v(R_{i-1}(t)) >0\,.   \]
The assumption (AK) on $K$ ensures that $\emph{I}_i < 0$, thus, on the other hand, $\mu_i(t)(R_i (t))^2 \emph{I}_i  <0$.
An analogous argument implies that, if $i$ such that $\mu_i(t)=2$, then $\emph{I}_i > 0$ and $2(R_i (t))^2 \emph{I}_i  >0$. These considerations lead immediately to
\begin{equation}\label{stuck[I]}
-\sum_{i=1}^{N-1}\mu_i(t)(R_i(t))^2\emph{I}_i < 0\,.
\end{equation}
Let us now focus on $-\sum_{i=1}^{N-1}\mu_i(t)R_i(t)\emph{II}_i$. In this case, we would like to obtain an upper bound in terms of $TV[\rho^N(t,\,\cdot)]$ and for this purpose we need to estimate $| II_i |$. We recall that $K'$ is locally Lipschitz and that $v(\rho)\in [0,v_{max}]$. The former in particular implies that $K'$ has finite Lipschitz constant on the compact interval $[-2\mathrm{meas}(\mathrm{supp}(\bar\rho)),2\mathrm{meas}(\mathrm{supp}(\bar\rho))]$, we name such a constant $L=L(\bar\rho)$. We get
\begin{align*}
|\emph{II}_i| &=   R_i(t)|v(R_i(t))| \left| -\sum_{j \neq i,\,i+1} \big(K'(x_{i+1}(t) - x_j(t)) - K'(x_i(t) - x_j(t))  \big)  -2K'(x_{i+1}(t) - x_i(t))\right| \\
&\leq R_i(t)\,L\,v_{max} \frac{N-2}{N} \frac{1}{R_i(t)} + 2v_{max}\,L\frac{1}{N} \leq L\,v_{max}\,,
\end{align*}
and this gives
\begin{equation}\label{stuck[II]}
\left|  -\sum_{i=1}^{N-1}\mu_i(t)R_i(t)\emph{II}_i  \right| \leq L\,v_{max}\left| \sum_{i=1}^{N-1}\mu_i(t)R_i(t) \right| \leq L\,v_{max}\,TV[\rho^N(t,\,\cdot)].
\end{equation}
We can now focus on $\dot{R}_0$ and $\dot{R}_N$. Since the setting is symmetric, we only present the argument for $\mu_0(t)\dot{R}_0$ and leave the one for $\mu_N(t)\dot{R}_N$ to the reader. Since $\mu_0(t)\neq 0$ only if $R_1(t)>R_0(t)$, without restriction we can assume $(v(R_1) - v(R_0)) \leq 0$ and can compute
\begin{align*}
\mu_0\dot{R}_0 &= \mu_0R_0 [R_0v(R_1)\sum_{j>1}\big(K'(x_1-x_j) - K'(x_0-x_j)\big) +2R_0v(R_0)K'(x_1-x_0)] \\
&\quad + \mu_0(R_0)^2(v(R_1) - v(R_0))\sum_{j>1} K'(x_0 - x_j) \\
&\leq  \mu_0R_0 [R_0v(R_1)\sum_{j>1}\big(K'(x_1-x_j) - K'(x_0-x_j)\big) +2R_0v(R_0)K'(x_1-x_0)]\,.
\end{align*}
Moreover,
\[ \left| R_0v(R_1)\sum_{j>1}\big(K'(x_1-x_j) - K'(x_0-x_j)\big) +2R_0v(R_0)K'(x_1-x_0)\right| \leq v_{max}\,L\frac{N-1}{N} + \frac{2v_{max}\,L}{N}\,.  \]
In particular, $\mu_0\dot{R}_0 \leq (3CL) R_0$ and
\begin{equation}\label{primoeultimotermine}
\mu_0\dot{R}_0 + \mu(R_N)\dot{R}_N \leq 3v_{max}\,L\, (R_0 + R_N) \leq 3v_{max}\,L\, TV[\rho^N(t,\,\cdot)]\,.
\end{equation}
By putting together~\eqref{stuck[I]},~\eqref{stuck[II]} and~\eqref{primoeultimotermine} we get estimate~\eqref{dTVlimitata} and~\eqref{TV} follows as a consequence of Gronwall Lemma.
\end{proof}

We now prove the equi-continuity w.r.t. time with respect to the $1$-Wasserstein distance for $\rho^N$.

\proofof{Proposition~\ref{continuitytime}}
Assume without loss of generality that $0< s<t < T$.
Our goal then is to investigate the continuity in time of the discrete density $\rho^N$ with respect to the $1$-Wasserstein distance. We exploit the well known relation between the $1$-Wasserstein distance of two probability measures and the $L^1$ distance of their respective pseudo inverse functions. More precisely, for any two probability measures $\mu,\,\nu$ the following identity holds
\[    d_{1} (\mu,\,\nu)  =  \| X_{\mu} - X_{\nu}\|_{L^1([0,\,1])}, \]
where $X_\mu$ and $X_\nu$ are the pseudo inverses of the cumulative distribution functions of $\mu$ and $\nu$ respectively.
The assertion of the proposition will follow once we prove that there exists a constant $C>0$ independent of $N$ such that
\[  \|  X_{\rho^N(t,\cdot)} - X_{\rho^N(s,\,\cdot)} \|_{L^1([0,1])} < C|t-s|, \]
for all $s,\,t \in (0,T)$.
By the definition of $\rho^N$ we can explicitly compute
\[ X_{\rho^N(t,\,\cdot)}(z) = \sum_{i=0}^{N-1} \left(x_i^N(t) + \left(z-i\frac{1}{N}\right) \frac{1}{R_i^N(t)}\right) \textbf{1}_{[i\frac{1}{N},\,(i+1)\frac{1}{N})}(z)\,. \]
Therefore,
\begin{align*}
d_{1}\big( \rho^N(t,\,\cdot), \rho^N(s,\,\cdot) \big) &= \| X_{\rho^N(t,\,\cdot)} - X_{\rho^N(s,\,\cdot)}  \|_{L^1([0,\,1])} \\
&\leq \sum_{i=0}^{N-1} \int_{i/N}^{(i+1)/N} \left| x_i^N(t) - x_i^N(s) + \left(z - \frac{i}{N} \right) \left(\frac{1}{R_i^N(t)} - \frac{1}{R_i^N(s)} \right)  \right| dz \\
&\leq \sum_{i=0}^{N-1} \frac{1}{N} |x_i^N(t) - x_i^N(s)| +  \sum_{i=0}^{N-1} \left|\frac{1}{R_i^N(t)} - \frac{1}{R_i^N(s)} \right|  \int_{i/N}^{(i+1)/N} \left(z - \frac{i}{N} \right)dz \\
&= \sum_{i=0}^{N-1} \frac{1}{N} |x_i^N(t) - x_i^N(s)|  + \sum_{i=0}^{N-1} \frac{1}{2N^2} \int_s^t \left| \frac{d}{d\tau} \frac{1}{R_i^N(\tau)}\right| d\tau \\
&\leq 3 \sum_{i=0}^{N} \frac{1}{N} \int_s^t \left|\dot{x}_i^N(\tau)\right| d\tau\,,
\end{align*}
where in the last inequality we used that
\[ \left| \frac{d}{d\tau} \frac{1}{R_i^N(\tau)}  \right| = N |\dot{x}_{i+1}^N(\tau) - \dot{x}_i^N(\tau)| \leq N |\dot{x}_{i+1}^N(\tau)| + N|\dot{x}_i^N(\tau)|\,. \]
Notice that we can control $|\dot{x}_i^N(\tau)|$ uniformly in $N$ and in $\tau$. Indeed, recalling the assumption (AK), setting $L$ as the Lipschitz constant of $K'$ on the interval $[-2\mathrm{meas}(\mathrm{supp}(\bar\rho)),2\mathrm{meas}(\mathrm{supp}(\bar\rho))]$ as in the proof of Proposition \ref{totalvariation}, we have
\[  |\dot{x}_i^N(\tau)| = \frac{1}{N}\left|-v(R_i(t)) | \sum_{j>i} K'(x_i - x_j) - v(R_{i-1})\sum_{j<i} K'(x_i - x_j) \right| \leq \frac{2LN v_{max}}{N} = 2L v_{max},, \]
which gives
\[ d_{1}\big( \rho^N(t,\,\cdot), \rho^N(s,\,\cdot) \big) \leq 6L v_{max}\,|t-s| \sum_{i=0}^N \frac{1}{N} \leq 12L v_{max} |t-s|,  \]
and~\eqref{contime} is proven.
\end{proof}

\section{Consistency of the many particle scheme: convergence to entropy solution}\label{sec:consi}

In this section we show that the limit $\rho$ obtained in Section \ref{sec:convergence} satisfies the entropy condition in the sense of Definition~\ref{solentropicadef}. Moreover, we can prove that $\rho$ is the unique entropy solution of the Cauchy problem
\begin{equation}
\left\lbrace\begin{array}{ll}
\partial_t \rho = \partial_x(\rho v(\rho) K' \ast \rho)\quad t \in (0,T],  \\
\rho(0,\cdot) = \bar{\rho}.
\end{array} \right.
\end{equation}
The first step consists in showing that the discrete densities satisfy an analogue version of the entropy condition. In order to do that, we need to introduce another sequence of approximations of the solution $\rho$, namely the $N$-empirical measure
\[ \hat{\rho}^N(t,x) := \frac{1}{N} \sum_{i=0}^N \delta_{x_i(t)}(x).\]
In the next lemma we show that $\hat{\rho}^N$ and $\rho^N$ are arbitrarily close in the $1$-Wasserstein distance, which implies that $\hat{\rho}^N$ converge up to a subsequence to the same limit $\rho$ obtained in the previous section.

\begin{lemma}\label{lem:empirical1}
  For all $N\in \mathbb{N}$, we have
  \[d_1(\rho^N(t,\cdot),\hat{\rho}^N(t))\leq \frac{C}{N},\]
  for some constant $C$ only depending on $\bar\rho$.
\end{lemma}

\begin{proof}
  In view of the standard isometric mapping between the $1$-Wasserstein space of probability measures and the convex cone of non-decreasing functions in $L^1([0,1])$, similarly to the proof of proposition \ref{continuitytime}, we have
\begin{align*}
     & d_1(\rho^N(t,\cdot),\hat{\rho}^N(t))\leq \sum_{i=0}^{N-1} \int_{i/N}^{(i+1)/N}\left| \left(z-i\frac{1}{N}\right) \frac{1}{R_i^N(t)}\right|\, dz\\
     & \ = \frac{1}{2N}\sum_{i=0}^{N-1}(x^N_{i+1}(t)-x^N_i(t)) = \frac{1}{2N}\left(x^N_{N}(t)-x^N_{0}(t)\right)\\
     & \ \leq \frac{1}{2N} \mathrm{meas}(\mathrm{supp}(\bar\rho)),
  \end{align*}
  which proves the assertion.
\end{proof}

\begin{remark}\label{rem:empirical}
  \emph{Let $W\in C(\R)$ be even and locally Lipschitz. Then, there exists a constant $C>0$ depending only on $\bar\rho$ such that
  \[\sup_{t\geq 0}\|W\ast \rho^N(t,\cdot)-W\ast \hat{\rho}^N(t)\|_{L^1}\leq \frac{C}{N},\]
  for all $N\in \mathbb{N}$. To prove this, let $\gamma^N_o(t)$ be an optimal plan between $\rho^N(t,\cdot)$ and $\hat{\rho}^N(t,\cdot)$ with respect to the cost $c(x)=|x|$. We then estimate, for all $t\geq 0$,
  \begin{align*}
     & \|W\ast \rho^N-W\ast \hat{\rho}^N\|_{L^1(\R)} = \int_\R \left|\int_\R W(x-y)\, d\rho^N(t,\cdot)(y)-\int_\R W(x-y)\,d\hat{\rho}^N(t)(y)\right|\, dx\\
     & \ =\int_\R \left|\iint_{\R^2} \left(W(x-y)-W(x-z)\right)\, d\gamma_0^N(t)(y,z)\right|\, dx\\
     & \ \leq C \int_\R \iint_{\R^2}|y-z|d\gamma_0^N(t)(y,z),dx,
  \end{align*}
  where we have used that the supports of $\rho^N$ and $\hat{\rho}^N$ are contained in $\mathrm{supp}(\bar\rho)$ which is bounded and independent of time. By definition of $1$-Wasserstein distance we therefore have
  \[\|W\ast \rho^N-W\ast \hat{\rho}^N\|_{L^1(\R)} \leq C d_1(\rho^N(t,\cdot),\hat{\rho}^N(t)) \leq \frac{\tilde{C}}{N},\]
  for some suitable constant $\tilde{C}>0$ in view of Lemma \ref{lem:empirical1}.}
\end{remark}

Our next goal is to prove that the entropy inequality
\[0 \leq \int_0^T \int_{\R} |\rho^N - c|\varphi_t -  \mathrm{sign}(\rho^N - c)[(f(\rho^N) - f(c))K' \star \hat{\rho}^N \varphi_x - f(c)K'' \ast \hat{\rho}^N \varphi]dxdt\]
holds for every non-negative test function $\varphi$ with compact support in $\mathcal{C}^{\infty}_c ((0,+\infty)\times \R)$, every constant $c\geq 0$ and every $N$ \emph{large enough}. Such a goal, which requires some tedious calculations, is however not enough to prove that the limit $\rho$ of the previous section is an entropy solution because the of the discontinuity of the sign function in the above inequality, which does not allow to pass to the limit for $\rho^N\rightarrow \rho$ almost everywhere and in $L^1$. To bypass this problem we shall then introduce a $\delta$-regularization of the sign function in order to first let $N\rightarrow +\infty$ and then $\delta\searrow 0$. In the last part of the section we prove the uniqueness of entropy solutions, which allows to conclude that the whole approximating sequence $\rho^N$ converges to $\rho$, thus completing the proof of our main Theorem~\ref{main}.


\begin{lemma}\label{entropiarhoN}
For every non negative $\varphi \in C^{\infty}_c ((0,+\infty)\times \R),\,c\geq 0$ and $N \in \N$ the following inequality holds
\begin{equation}\label{entropiaN}
\liminf_{N\rightarrow +\infty}\int_0^T \int_{\R} |\rho^N - c| \varphi_t - \mathrm{sign}(\rho^N -c)[(f(\rho^N) - f(c))K' \ast \hat{\rho}^N \varphi_x - f(c)K'' \ast \hat{\rho}^N \varphi] dx dt \geq 0.
\end{equation}
\end{lemma}
\begin{proof}
Let $T>0$ such that $\mathrm{supp}\varphi\subset[0,T]\times \R$. The basic idea of the proof is rather simple, although the computations are quite technical: we need to rewrite the left hand side of the inequality so that it is possible to isolate a term with positive sign and then show that the remaining terms give negligible contributions as $N \to \infty$.
By definition of $\rho^N$ and $\hat{\rho}^N$ we obtain
\[ \int_0^T \int_{\R} |\rho^N - c| \varphi_t - \mathrm{sign}(\rho^N -c)[(f(\rho^N) - f(c))K' \ast \varphi_x - f(c)K'' \ast \rho^N \varphi] dx dt = B.T._1 + \sum_{i=0}^{N-1} I_i + \sum_{i=0}^{N-1} II_i, \]
where
\begin{align*}
& I_i := \int_0^T \int_{x_i}^{x_{i+1}} |R_i^N -c| \varphi_t\,dxdt ,\\
& II_i := - \int_0^T \int_{x_i}^{x_{i+1}} \mathrm{sign}(R_i^N - c)(f(R_i^N) - f(c))K' \ast \hat{\rho}^N \varphi_x\,dxdt  \\
&\hspace*{1.2cm}+ \int_0^T \int_{x_i}^{x_{i+1}} f(c) \mathrm{sign}(R_i^N - c) K'' \ast \hat{\rho}^N \varphi\, dxdt, \\
& B.T._1 := \int_0^T \int_{-\infty}^{x_0} c\varphi_t - f(c)[K' \ast \hat{\rho}^N \varphi_x + K'' \ast \hat{\rho}^N \varphi]\,dxdt  \\
&\hspace*{1.2cm}+ \int_0^T \int_{x_N}^{\infty} c\varphi_t - f(c)[K' \ast \hat{\rho}^N \varphi_x + K'' \ast \hat{\rho}^N \varphi]\,dxdt.
\end{align*}
For simplicity of notation we set $S^N_i : = \mathrm{sign}(R_i^N - c)$ and we omit the dependence on $N$ and $t$ wherever it is clear from the context.
Integrating by parts and recalling the definition of $\hat{\rho}^N$ and the expression for $\dot{R}_i$, we can rewrite $I_i$ as
\begin{align*}
I_i=& \int_0^T S_i R_i (\dot{x}_{i+1} - \dot{x}_i) \left(\intmed_{x_{i}}^{x_{i+1}} \varphi(t,x) dx   - \varphi(t,x_{i+1})\right)dt  \\
&+ \int_0^T S_i [R_i (\dot{x}_{i+1} - \dot{x}_i) \varphi(t,x_{i+1})  - (R_i -c) (\dot{x}_{i+1}\varphi(t,x_{i+1}) - \dot{x}_i \varphi(t,x_i))]dt,
\end{align*}
and $II_i$ as
\begin{align*}
II_i =& -\int_0^T S_i\frac{(f(R_i) - f(c))}{N}\sum_{j=0}^N (K'(x_{i+1} - x_j)\varphi(t,x_{i+1}) - K'(x_i - x_j)\varphi(t,x_i))dt \\
&+ \int_0^T S_i \frac{f(R_i)}{N} \sum_{j=0}^N \int_{x_i}^{x_{i+1}} K''(x-x_j)\varphi(t,x)dxdt\,.
\end{align*}
Then the sum $I_i + II_i$ becomes
\[I_i + II_i = A^1_i + A^2_i + Z_i, \]
where we set
\begin{align*}
A^1_i &= \int_0^T S_i R_i (\dot{x}_{i+1} - \dot{x}_i) \left(\intmed_{x_{i}}^{x_{i+1}} \varphi(t,x) dx   - \varphi(t,x_{i+1})\right)dt, \\
A^2_i &=  \int_0^T S_i \frac{f(R_i)}{N} \sum_{j=0}^N \int_{x_i}^{x_{i+1}} K''(x-x_j)\varphi(t,x)dxdt,
\end{align*}
and
\begin{align*}
Z_i = &- \sum_{i=0}^{N-1} \int_0^T S_i\varphi(t,x_{i+1})[R_i \dot{x}_i + \frac{f(R_i)}{N}\sum_{j=0}^N K'(x_{i+1}-x_j)]dt \\
&+ \sum_{i=0}^{N-1} \int_0^T S_i\varphi(t,x_{i+1})[c \dot{x}_{i+1} + \frac{f(c)}{N}\sum_{j=0}^N K'(x_{i+1}-x_j)]dt \\
&+ \sum_{i=0}^{N-1} \int_0^T S_i\varphi(t,x_{i})[R_i \dot{x}_i + \frac{f(R_i)}{N}\sum_{j=0}^N K'(x_{i}-x_j)]dt \\
&- \sum_{i=0}^{N-1} \int_0^T S_i\varphi(t,x_{i})[c \dot{x}_i + \frac{f(c)}{N}\sum_{j=0}^N K'(x_{i}-x_j)]dt.
\end{align*}
By performing a summation by parts, we get
\begin{align*}
\sum_{i=0}^{N-1} Z_i &= B.T._2 + \sum_{i=1}^{N-1} \int_0^T \varphi(t,x_i) S_i\left(R_i \dot{x}_i +\frac{f(R_i)}{N} \sum_{j=0}^N K'(x_i -x_j) \right)dt \\
&\quad - \sum_{i=1}^{N-1} \int_0^T \varphi(t,x_i) S_{i-1}\left(R_{i-1}\dot{x}_{i-1} + \frac{f(R_{i-1})}{N}\sum_{j=0}^N K'(x_i-x_j) \right)dt \\
&\quad +\sum_{i=1}^{N-1} \int_0^T \varphi(t,x_i)(S_{i-1}- S_i)\left(c\dot{x}_i + \frac{f(c)}{N} \sum_{j=0}^N K'(x_i-x_j) \right)dt \\
&= B.T._2 + B.T._3 + \sum_{i=1}^{N-2} (A_i^3 + A_i^4) + \sum_{i=1}^{N-1}B_i.
\end{align*}
where $B.T._2$ and $B.T._3$ regard the external particles. More precisely, $B.T._2= B.T._{21}+ B.T._{22}$, where
\begin{align*}
B.T._{21} =& c \int_0^T \varphi(t,x_N) S_{N-1} \frac{v(c)-v(R_{N-1})}{N} \sum_{j=0}^N K'(X_N - x_j)dt \\
&  -c \int_0^T \varphi(t,x_0) S_{0}\frac{v(c)-v(R_{0})}{N}\sum_{j=0}^N K'(X_0 - x_j)dt, \\
B.T._{22} =& \int_0^T \varphi(t,x_0)S_{0}R_0 \left(\dot{x}_0 + \frac{v(R_0)}{N}\sum_{j=0}^N K'(X_0 - x_j)\right)dt \\
&- \int_0^T \varphi(t,x_N)S_{N-1} R_{N-1} \left(\dot{x}_{N-1} + \frac{v(R_{N-1})}{N}\sum_{j=0}^N K'(X_N - x_j)\right)dt,
\end{align*}
and $B.T._3$ corresponds to
\begin{align*}
B.T._3 =& \int_0^T \varphi(t,x_{N-1}) S_{N-1}\left(R_{N-1}\dot{x}_{N-1} + \frac{f(R_{N-1})}{N} \sum_{j=0}^N K'(x_{N-1}-x_j)\right)dt \\
&- \int_0^T \varphi(t,x_{0})S_0 \left( R_{0}\dot{x}_{0} + \frac{f(R_{0})}{N} \sum_{j=0}^N K'(x_{1}-x_j)\right)dt.
\end{align*}
The terms $A_i^3,\,A_i^4$ and $B_i$ regards, instead, the internal particles and they are defined as follows
\begin{align*}
A_i^3 =& \int_0^T \varphi(t,x_i) S_i \frac{f(R_i)}{N}\sum_{j=0}^N [K'(x_i-x_j)-K'(x_{i+1}-x_j)]dt, \\
A_i^4 =& \int_0^T (\varphi(t,x_i) - \varphi(t,x_{i+1}))S_i R_i \left( \dot{x}_i + \frac{v(R_i)}{N} \sum_{j=0}^N K'(x_{i+1}-x_j) \right)dt,\\
B_i =& \int_0^T \varphi(t,x_i)(S_{i-1}- S_i))\left(c\dot{x}_i + \frac{f(c)}{N} \sum_{j=0}^N K'(x_i-x_j) \right)dt.
\end{align*}
Summarizing, we can rewrite $B.T._1 + \sum_{i=0}^{N-1} (I_i + II_i)$ as
\[ B.T._1 + B.T_{21} + B.T._{22} + B.T._3 + \sum_{i=0}^{N-1} (A^1_i + A^2_i) + \sum_{i=1}^{N-2} (A^3_i + A^4_i) + \sum_{i=1}^{N-1} B_i,  \]
then estimate \eqref{entropiaN} follows if we prove that such sum is non negative when $N \gg 1$, and this can be done by showing that
\begin{equation}\label{partepositiva}
B.T._1 + B.T._{21} + \sum_{i=1}^{N-1}B_i > 0,
\end{equation}
while
\begin{equation}\label{ordine1/N}
\left| B.T._{22} + B.T._3 + \sum_{i=0}^{N-1}(A_i^1 + A_i^2) + \sum_{i=1}^{N-2} (A_i^3 + A_i^4) \right| \leq \frac{C}{N}
\end{equation}
for a positive constant $C= C(\varphi,K,\bar{\rho},v,T)$.
The remaining part of the proof is devoted to show the validity of \eqref{partepositiva} and \eqref{ordine1/N}.
We focus first on \eqref{partepositiva}. Integrating by parts, recalling that $\varphi(0,\cdot)=\varphi(T,\cdot)=0$, $\varphi(t,\cdot) \geq 0$ and the assumption (AK), we immediately obtain
\begin{equation}\label{BT1}
B.T._1 = -\frac{f(c)}{N} \int_0^T \left(\varphi(t,x_N)\sum_{j=0}^N K'(x_N-x_j) + \varphi(t,x_0)\sum_{j=0}^N K'(x_0-x_j)\right) > 0.
\end{equation}
Because of the monotonicity of $v$ (see (Av)), for all times $t$ we know that
\[ S_{0}(t)(v(c) - v(R_{0}(t))) \geq 0,\quad \mbox{ and }\quad S_{N-1}(t)(v(c) - v(R_{N-1}(t))) \geq 0  \]
thus, recalling again (AK), we deduce
\begin{equation}\label{BT21}
B.T._{21} \geq 0.
\end{equation}
Let us now consider the generic term $B_i$. Substituting the expression of $\dot{x}_i$, we get
\[ B_i = \int_0^T\varphi(t,x_i) (S_{i-1} - S_i)\left[\frac{v(c)-v(R_i)}{N} \sum_{j>i} K'(x_i-x_j) + \frac{v(c)-v(R_{i-1})}{N} \sum_{j<i} K'(x_i-x_j) \right]dt.   \]
Now, if $R_i(t),R_{i-1}(t)$ are both strictly bigger than $c$ or strictly smaller than $c$, then $S_{i-1}(t) - S_i(t)=0$. Otherwise, since $v$ is decreasing (and we assume $\mathrm{sign}(0)=0$), we get
\begin{align*}
R_i(t) \geq c \geq R_{i-1}(t) &\quad\Rightarrow\, -2 \leq S_{i-1}(t) - S_i(t)\leq 0, \quad v(c)-v(R_i(t)) \geq 0, \quad v(c)-v(R_{i-1}) \leq 0 \\
R_{i-1}(t) \geq c \geq R_i(t) &\quad\Rightarrow\, 0 \leq S_{i-1}(t) - S_i(t) \leq 2, \quad v(c)-v(R_i(t)) \leq 0, \quad v(c)-v(R_{i-1}) \geq 0
\end{align*}
and, recalling that $K'$ is symmetric and $K'(x)>0$ if $x>0$, for all times holds
\[ (S_{i-1}- S_i)\left[\frac{v(c)-v(R_i)}{N} \sum_{j>i} K'(x_i-x_j) + \frac{v(c)-v(R_{i-1})}{N} \sum_{j<i} K'(x_i-x_j) \right] \geq 0.  \]
In particular, $B_i \geq 0$ and
\begin{equation}\label{Bi}
\sum_{i=1}^{N-1} B_i \geq 0.
\end{equation}
Then estimate \eqref{partepositiva} is a direct consequence of \eqref{BT1}, \eqref{BT21} and \eqref{Bi}. Let us consider now \eqref{ordine1/N}. First of all, observe that Lemma \ref{MaxPrinc} ensures that the support of $\rho^N$ is always contained in the support of $\bar{\rho}$, then since we assume $K'$ locally Lipschitz then there exists a constant $L>0$ such that
\[ L=\sup\left\{ |K''(x)|\,,\,\, x\in[-(\bar{x}_{max}-\bar{x}_{min}),(\bar{x}_{max}-\bar{x}_{min})]\right\}.  \]
Since the argument is quite technical, it is more convenient to split the left hand side of \eqref{ordine1/N} in three parts:
\[ \Gamma_1 = B.T._{22} + B.T._3 + A_0^2 + A_{N-1}^2,\quad \Gamma_2 = \sum_{i=0}^{N-1} A_i^1 + \sum_{i=1}^{N-2} A_i^4,\quad \Gamma_3 = \sum_{i=1}^{N-2} (A_i^2 + A_i^3).  \]
Recalling that $K',\,\varphi$ and $v$ are uniformly bounded and Lipschitz, we get
\begin{align}\label{Gamma1}
\notag|\Gamma_1| \leq\, &4 L\,\|\varphi\|_{L^{\infty}} \|v\|_{L^{\infty}} \int_0^T (R_0(x_1 - x_0) + R_{N-1}(x_{N}-x_{N-1}))dt  \\
&+ 2 L\,\|v\|_{L^{\infty}}Lip[\varphi] \int_0^T R_{N-1}(x_N - x_{N-1})dt \leq \frac{C(\varphi,v,L,T)}{N}
\end{align}
Then, inserting the expression of $\dot{x}_i$, we can rearrange $\Gamma_2$ in such a way that
\begin{align*}
|\Gamma_2| &\leq 3\sum_{i=0}^{N-1} \int_0^T R_i \left|\intmed_{x_i}^{x_{i+1}}\varphi(t,x) - \varphi(t,x_{i+1})\right|\frac{|v(R_{i+1})-v(R_i)|}{N}\sum_{j=0}^N |K'(x_{i+1}-x_j)| dt \\
&\quad + \sum_{i=0}^{N-1} \int_0^T R_i \left|\intmed_{x_i}^{x_{i+1}}\varphi(t,x) - \varphi(t,x_{i+1})\right| \frac{v(R_i)}{N} \sum_{j=0}^N |K'(x_i-x_j)- K'(x_{i+1}-x_j)|dt \\
&\quad + \sum_{i=1}^{N-2} \int_0^T R_i |\varphi(t,x_i) - \varphi(t,x_{i+1})|\frac{|v(R_{i-1})-v(R_i)|}{N}\sum_{j=0}^N |K'(x_{i}-x_j)|dt \\
&\quad + \sum_{i=1}^{N-2} \int_0^T R_i |\varphi(t,x_i) - \varphi(t,x_{i+1})|\frac{v(R_i)}{N} \sum_{j=0}^N |K'(x_i-x_j)- K'(x_{i+1}-x_j)|dt
\end{align*}
and using the Lipschitz and the uniform regularity of $K',\,\varphi,\,v$, estimate \eqref{TV} and the uniform bound on the support of $\rho^N$, it is easy to see that
\begin{align}\label{Gamma2}
\notag|\Gamma_2| \leq& \,4L\, Lip[\varphi]Lip[v] TV[\bar{\rho}] \int_0^T e^{Ct} \sum_{i=0}^{N-1} R_i (x_{i+1}-x_i)dt \\
& + 2 L\,\|v\|_{L^{\infty}} Lip[\varphi] \int_0^T \sum_{i=0}^{N-1} R_i (x_{i+1}-x_i)^2 dt \leq \frac{C(\varphi,v,K,\bar{\rho},T)}{N}.
\end{align}
It remains to show that also $\Gamma_3$ vanishes as $N \to \infty$. In this case, the uniform bound on $K''$ implies
\begin{align}\label{Gamma3}
\notag |\Gamma_3| &\leq \sum_{i=1}^{N-2} \int_0^T \frac{|f(R_i)|}{N} \int_{x_i}^{x_{i+1}}|\varphi(t,x)-\varphi(t,x_i)| \sum_{j=0}^N |K''(x-x_j)|dxdt \\
&\leq L\,\|v\|_{L^{\infty}} Lip[\varphi] \int_0^T R_i \int_{x_i}^{x_{i+1}} (x-x_i)dxdt \leq \frac{C(\varphi,v,K,\bar{\rho},T)}{N}.
\end{align}
Finally, by putting together \eqref{Gamma1},\eqref{Gamma2} and \eqref{Gamma3}, we obtain \eqref{ordine1/N} and, recalling also \eqref{partepositiva}, \eqref{entropiaN}.
\end{proof}

We are now in position to prove that the large particle limit $\rho$ that we obtained in the previous section is an entropy solution for the PDE.

\begin{lemma}\label{entropiarho}
Let $\rho$ be the limit of $\rho^N$ up to a subsequence. For every non negative $\varphi \in C^{\infty}_c ([0,+\infty)\times \R)$ and $c\geq 0$, one has
\begin{equation}\label{entropia}
0 \leq \int_\R |\bar\rho - c|\varphi(0,x)dx + \int_0^{+\infty} \int_{\R} |\rho - c| \varphi_t - \mathrm{sign}(\rho -c)[(f(\rho) - f(c))K' \ast \rho\, \varphi_x - f(c)K'' \ast \rho\, \varphi] dx dt.
\end{equation}
\end{lemma}
\begin{proof}
Let $T>0$ be such that $\mathrm{supp}(\varphi)\subset [0,T)$. Roughly speaking, the statement holds provided we can show that it is possible to pass to the limit as $N \to \infty$ in the inequality~\eqref{entropiaN}. More precisely,
we need to prove the following
\begin{align*}
&  \lim_{N \to \infty} \int_\R |\rho^N(0,x) - c|\varphi(0,x)dx = \int_\R |\bar\rho - c|\varphi(0,x)dx\\
& \lim_{N \to \infty} \int_0^T \int_{\R} |\rho^N-c|\,\varphi_t\, dx dt = \int_0^T \int_{\R} |\rho-c|\,\varphi_t \, dx dt \\
& \lim_{N \to \infty} \int_0^T \int_{\R} \mathrm{sign}(\rho^N-c)(f(\rho^N)-f(c))K'\ast \hat{\rho}^N\,\varphi_x \, dx dt\\
 & \qquad = \int_0^T \int_{\R} \mathrm{sign}(\rho-c)(f(\rho)-f(c))K'\ast \rho \,\varphi_x \, dx dt\\
& \lim_{N \to \infty} \int_0^T \int_{\R} f(c) \mathrm{sign}(\rho^N - c)K'' \ast \hat{\rho}^N\,\varphi \, dx dt = \int_0^T \int_{\R} f(c)   \mathrm{sign}(\rho - c)K'' \ast \rho\,\varphi\, dx dt
\end{align*}
The first two limits are immediate in view of the strong $L^1$-convergence of $\rho^N(0,x)$ to $\bar\rho$ and to the convergence of $\rho^N$ to $\rho$ almost everywhere in $L^1([0,T] \times \R)$ respectively.
Notice now that the continuity of $f$ ensures the continuity of the function $h(\mu):= \mathrm{sign}(\mu - c)(f(\mu)-f(c))$. We have
\begin{align*}
\int_0^T \int_{\R}& [\mathrm{sign}(\rho^N-c)(f(\rho^N)-f(c))K'\ast \hat{\rho}^N - \mathrm{sign}(\rho-c)(f(\rho)-f(c))K'\ast \rho ]\,\varphi_x\,dx dt\\
&= \int_0^T \int_{\R}(h(\rho)-h(\rho^N))K'\ast \rho\, \varphi_x\,dx dt + \int_0^T \int_{\R}  h(\rho^N) K'\ast(\rho -\hat{\rho}^N)\,\varphi_x\,dx dt,\\
\end{align*}
then the regularity of $h$ and $K'$ required in the assumptions (Av) and (AK), the convergence of $\rho^N$ to $\rho$ almost everywhere in $[0,T] \times \R$ and the strong $L^1$-convergence of $K' \ast \hat{\rho}^N$ to $K' \ast \rho$ established in Remark \ref{rem:empirical} imply that
\begin{equation}\label{bla2}
\int_0^T \int_{\R}|[(h(\rho)-h(\rho^N))K'\ast \rho +  h(\rho^N) K'\ast(\rho -\hat{\rho}^N)]\,\varphi_x|\,dx dt \to 0
\end{equation}
as $N$ tends to $+\infty$. Concerning the fourth limit, instead, we can see that
\begin{align*}
\int_0^T \int_{\R}& f(c)[\mathrm{sign}(\rho^N-c)K''\ast \hat{\rho}^N - \mathrm{sign}(\rho-c)K''\ast\rho]\,\varphi\,dx dt \\
&= \int_0^T \int_{\R} f(c)\mathrm{sign}(\rho^N-c)K''\ast (\hat{\rho}^N - \rho)\,\varphi\,dx dt \\
&\quad + \int_0^T \int_{\R} f(c)(\mathrm{sign}(\rho^N-c) -\mathrm{sign}(\rho-c))K''\ast \rho\,\varphi\,dx dt.
\end{align*}
The first of the two terms can be handled as above. By using Remark \ref{rem:empirical} and Lemma \ref{lem:empirical1}, we get 
\begin{equation}\label{bla3}
\int_0^T \int_{\R} |f(c)\mathrm{sign}(\rho^N-c)K''\ast (\hat{\rho}^N - \rho)\,\varphi\,dx| dt \leq \frac{C(K'',\|f\|_{\infty},\varphi)}{N}\,.
\end{equation}
On the other hand, passing to the limit in the terms including the difference $\mathrm{sign}(\rho^N-c)- \mathrm{sign}(\rho-c)$ is less straightforward because of the discontinuity of the sign function.
Let us then focus on the proof of
\[ \lim_{N \to \infty} \int_0^T \int_{\R} f(c)(\mathrm{sign}(\rho^N-c)-\mathrm{sign}(\rho-c))K''\ast \rho\,\varphi\,dxdt = 0\,. \]
In order to get rid of the discontinuity, we need to introduce two smooth approximations of the \emph{sign} function, we call them $\eta_{\delta}^{\pm}$, so that
\[ \mathrm{sign}(z) - \eta_{\delta}^+(z) \geq 0 \quad \mbox{ and } \quad \mathrm{sign}(z)-\eta_{\delta}^-(z) \leq 0.  \]
Let us recall that the regularity of $K$ ensures the existence of a constant $L>0$ such that $|K''(z)| \leq L$ for every $z \in [-2\mathrm{meas}(\mathrm{supp}(\rho^N)), 2\mathrm{meas}(\mathrm{supp}(\rho^N))]$ and every $N$.
Then we can estimate
\begin{align*}
\int_0^T\int_{\R}f(c)\mathrm{sign}&(\rho^N-c) K''\ast \rho \varphi \\
&= \int_0^T\int_{\R}f(c)\mathrm{sign}(\rho^N-c)(K''-L)\ast\rho\,\varphi + \int_0^T\int_{\R}f(c)\mathrm{sign}(\rho^N-c)L\ast \rho\,\varphi \\
&\leq \int_0^T\int_{\R}f(c)\eta_{\delta}^+(\rho^N-c)(K''-L)\ast \rho\,\varphi + \int_0^T\int_{\R}f(c)\eta_{\delta}^-(\rho^N-c)L\ast \rho\,\varphi\,.
\end{align*}
where the inequality holds because
\begin{align*}
&(\mathrm{sign}(\rho^N-c)-\eta_{\delta}^+(\rho^N-c))(K''-L)\ast \rho \leq 0 \\
&(\mathrm{sign}(\rho^N-c)-\eta_{\delta}^-(\rho^N-c)) L\ast \rho \leq 0.
\end{align*}
Now, observe that
\begin{align*}
\lim_{N\to \infty}& f(c) \int_0^T \int_{\R} (\eta_{\delta}^+(\rho^N -c) - \eta_{\delta}^+(\rho -c)) (K''-L)\ast \rho\,\varphi\\
&\leq \lim_{N\to \infty} f(c)\int_0^T\int_{\R} |\eta_{\delta}^+(\rho^N -c) -\eta_{\delta}^+(\rho -c)| |(K''-L)\ast\rho\,\varphi| \\
&\leq \lim_{N\to \infty} f(c) 2L \| \varphi\|_{\infty} Lip[\eta_{\delta}^+] \int_0^T\int_{\R} |\rho^N-\rho| \\
&\leq C(L, \varphi, \eta_{\delta}^+) \lim_{N\to \infty} \| \rho^N - \rho\|_{L^1([0,T]\times \R)}= 0
\end{align*}
and in a similar way we get also
\[ \lim_{N \to \infty} \int_0^T\int_{\R} f(c)(\eta_{\delta}^-(\rho^N-c) - \eta_{\delta}^-(\rho-c))L\ast \rho\,\varphi = 0\,, \]
thus implying that
\begin{equation*}
\limsup_{N\to\infty} \int_0^T\int_{\R}f(c)\mathrm{sign}(\rho^N-c) K''\ast \rho\,\varphi \leq \int_0^T\int_{\R} f(c)[\eta_{\delta}^+ (\rho-c)(K''-L)\ast\rho + \eta_{\delta}^-(\rho-c)L\ast \rho]\,\varphi
\end{equation*}
Once here, the dominated convergence Theorem ensures that we can pass to the limit in $\delta$ to get
\[\limsup_{N \to \infty} f(c)\int_0^T\int_{\R} \mathrm{sign}(\rho^N-c) K'' \ast \rho\,\varphi \leq \int_0^T\int_{\R} \mathrm{sign}(\rho-c) K'' \ast \rho\,\varphi.\]
A symmetric argument provides the inverse inequality with $\liminf$ replacing $\limsup$, hence we obtain
\begin{equation}\label{limiteinN}
\lim_{N \to \infty} f(c)\int_0^T\int_{\R} (\mathrm{sign}(\rho^N-c)-\mathrm{sign}(\rho-c)) K'' \ast \rho\,\varphi = 0.
\end{equation}
The above argument, together with~\eqref{bla2}-\eqref{limiteinN}, implies estimate~\eqref{entropia}, and the proof is complete.
\end{proof}

We now tackle another crucial task for our result, namely the \emph{uniqueness of the entropy solution} for a fixed initial datum. To perform this task we rely on a stability result due to Karlsen and Risebro \cite{KarlsenRiebro}, that we report here for sake of completeness in an adapted version.

\begin{theorem}\label{KarlsenRiebro}
Let $f,P,Q$ be such that
\[ f \quad \mbox{ is locally Lipschitz, }\qquad P,Q \in W^{1,1}(\R) \cap \mathcal{C}(\R), \qquad P_x,Q_x \in L^{\infty}(\R), \]
and let $p,q \in L^\infty([0,T]; BV(\R))$ be respectively entropy solutions to
\[ \left\lbrace \begin{array}{ll}
&p_t = (f(p)P(x))_x \quad p(0,x)=p_0(x),\\
&q_t = (f(q)Q(x))_x \quad q(0,x)=q_0(x),
\end{array}\right.\]
where the initial data $(p_0,q_0)$ are in $L^1(\R) \cap L^{\infty}(\R) \cap BV(\R)$.
Then for almost every $t \in (0,T)$ one has
\begin{equation}\label{contrattivitasol}
\| p(t) - q(t)  \|_{L^1(\R)} \leq \| p_0 - q_0  \|_{L^1(\R)} + t(C_1 \| P-Q\|_{L^{\infty}(\R)} + C_2\| P-Q\|_{BV(\R)})
\end{equation}
where $C_1 = Lip[f] \min\{ \|P\|_{BV(\R)}, \|Q \|_{BV(\R)}\}$ and $C_2 = \| f\|_{L^{\infty}}$.
\end{theorem}

We are now ready to prove our main theorem.

\proofof{Theorem~\ref{main}}
The results in Theorem \ref{convergence} and Lemma \ref{entropiarho} imply that there exist a subsequence of $\rho^N$ converging almost everywhere on $[0,+\infty)\times \R$ and in $L^1_{loc}$ to an entropy solution $\rho$ to \eqref{CauchyProblem} in the sense of Definition \ref{solentropicadef}. Therefore, the proof of Theorem~\ref{main} is concluded once we show that $\rho$ is the unique entropy solution. We argue by contradiction. Assume that there exist two different functions $\rho$ and $\varrho$ satisfying Definition~\ref{solentropicadef} with $\rho(0,\cdot) = \varrho(0,\cdot) = \bar{\rho}$, then we can define two vector fields $P(x) = K' \ast \rho (x)$ and $Q(x)= K'\ast \varrho(x)$.
In order to apply Theorem~\ref{KarlsenRiebro} to $P$ and $Q$, let us check that all assumptions therein are satisfied.
First of all, $P$ and $Q$ are locally Lipschitz in $\R$ thanks to the assumption (AK), thus $P_x, Q_x \in L^{\infty}_{loc}(\R)$. Then, we observe that
\begin{align*}
|P(t,x) - Q(t,x)| &= \left| \int_{\R} K'(x-y)\rho(t,y)dy - \int_{\R} K'(x-y)\varrho(t,y)dy \right| \\
&\leq \int_{\R} |K'(x-y)(\rho(t,y) - \varrho(t,y))|dy \leq L_{\bar{\rho}} \| \rho - \varrho\|_{L^{\infty}([0,T];L^1(\R))},
\end{align*}
and
\begin{align*}
\int_{\R} |P_x(s,x) - Q_x(s,x)|dx &= \int_{\R} |K'' \ast \rho(t,x) - K'' \ast \varrho(t,x)| dx \\
&= \int_{\R} |K''\ast(\rho - \varrho)(t,x)|dx  \leq L_{\bar{\rho}} \|\rho - \varrho\|_{L^{\infty}([0,T];L^1(\R))},
\end{align*}
where $L_{\bar\rho}=\max\{\|K'\|_{L^\infty(I_{\bar\rho})},\|K''\|_{L^1(I_{\bar\rho})}\}$, and $I_{\bar\rho}=[-2\mathrm{meas}(\mathrm{supp}(\bar{\rho})), 2\mathrm{meas}(\mathrm{supp}(\bar{\rho}))]$.
As a consequence
\begin{align*}
& \| P-Q\|_{L^{\infty}([0,T] \times \R)} \leq L_{\bar{\rho}} \| \rho - \varrho\|_{L^{\infty}([0,T];L^1(\R))}\\
 & \| P-Q\|_{L^{\infty}([0,T] ; BV(\R))} \leq L_{\bar{\rho}} \| \rho - \varrho\|_{L^{\infty}(0,T;L^1(\R))}.
\end{align*}
By applying Theorem~\ref{KarlsenRiebro} to $\rho,\varrho, P$ and $Q$ we obtain
\begin{equation}\label{assurdo}
\| \rho(t) - \varrho(t)  \|_{L^1(\R)} \leq C(K,\bar{\rho}) t \| \rho(t) - \varrho(t) \|_{L^1(\R)}.
\end{equation}
Assume that there exists an open interval $(t_1,t_2)\subset [0,T]$ such that $\rho(t,\cdot)$ and $\varrho(t,\cdot)$ differ in $L^1(\R)$ on $t\in (t_1,t_2)$. Then, due to the fact that \eqref{eq:intro_PDE} is invariant with respect to time-translations, the inequality  ~\eqref{assurdo} implies
\begin{equation}\label{assurdovero}
\| \rho(t,\cdot) - \varrho(t,\cdot)  \|_{L^1(\R)} \leq C(K,\bar{\rho}) (t-t_1) \| \rho(t,\cdot) - \varrho(t,\cdot) \|_{L^1(\R)}\quad \forall\,t \in (t_1,t_2).
\end{equation}
Clearly can always consider $t\in (t_1,t_2)$ small enough such that $C(K,\bar{\rho}) (t-t_1)< 1$, but this is in contradiction with~\eqref{assurdovero}. In conclusion, $\rho(t,\cdot)\equiv \varrho(t,\cdot)$ on $[0,T]$ and the proof is complete.
\end{proof}

\section{Non-uniqueness of weak solutions and steady states}\label{sec:discussion}

The use of the notion of entropy solution in the present context is not merely motivated by the technical need of identifying a notion of solution (stronger than weak solutions) allowing to prove uniqueness. Similarly to what happens for scalar conservation laws, we prove that there are explicit examples of initial data in $BV$ for which there exists two weak solutions to the Cauchy problem \eqref{CauchyProblem}. 

For simplicity, we use
\[v(\rho)=(1-\rho)_{+}.\]
Consider the initial condition
\[\bar\rho(x)=\mathbf{1}_{[-1,-1/2]}+\mathbf{1}_{[1/2,1]}.\]
Clearly, the stationary function
\[\rho_s(t,x)=\mathbf{1}_{[-1,-1/2]}+\mathbf{1}_{[1/2,1]}\]
is a weak solution to \eqref{CauchyProblem} with initial condition $\bar\rho$. To see this, let $\varphi\in C^1_c([0,+\infty)\times \R)$. We have
\begin{align*}
 & \int_0^{+\infty}\int_\R\left[\rho_s\varphi_t + \rho_s v(\rho_s)K'\ast \rho\varphi_x\right] dx dt + \int_\R\bar\rho\varphi(0,x) dx \\
 & \ = \int_0^{+\infty}\frac{d}{dt}\left(\int_{[-1,-1/2]\cup [1/2,1]}\varphi dx\right)dt +\int_{[-1,-1/2]\cup [1/2,1]}\varphi(0,x) dx = 0.
\end{align*}

We now prove that $\rho_s$ is not an entropy solution, in that it does not satisfy the entropy condition in Definition \ref{solentropicadef}. Let $\psi\in C^\infty_c(\R)$ be a standard non-negative mollifier supported on $[-1/4,1/4]$. Let $T>0$ and consider the test function $\varphi(t,x)=\phi(x)\xi(t)$ with
\[\phi(x)=
\begin{cases}
\psi(x+1/2) & \hbox{if $-3/4\leq x\leq -1/4$} \\
\psi(x-1/2) & \hbox{if $1/4\leq x\leq 3/4$} \\
0 & \hbox{otherwise},
\end{cases}
\]
and $\xi\in C^\infty([0,+\infty))$ with $\xi(t)=1$ for $t\leq T$, $\xi(t)=0$ for $t\geq T+1$ and $\xi$ non-increasing. Let us set $c=1/2$, $I=[1/4,3/4]$, and compute
\begin{align*}
  & \int_\R |\rho_s -c|\phi dx +\int_0^{+\infty}\int_\R\left[|\rho_s-c|\phi(x)\xi'(t) -\mathrm{sign}(\rho_s-c)(f(\rho)-f(c))K'\ast\rho_s\phi'(x)\xi(t) \right.\\
  & \qquad \left.-f(c)K''\ast \rho_s \phi(x)\xi(t)\right]\,dxdt\\
  & \leq 2\int_I \varphi dx + \frac{1}{4}\int_0^{T+1}\xi(t)dt\left[\int_{(-I)\cap(-\infty,1/2]}K'\ast\rho_s \varphi_x dx - \int_{(-I)\cap[1/2,+\infty)}K'\ast\rho_s \varphi_x dx \right.\\
  & \quad \left.- \int_{I\cap(-\infty,1/2]}K'\ast\rho_s \varphi_x dx +\int_{I\cap[1/2,+\infty)}K'\ast\rho_s \varphi_x dx- \int_{(-I)\cup I}K''\ast \rho_s \varphi dx\right]\\
  & = 2\int_I \varphi dx - \frac{1}{4}\int_0^{T+1}\xi(t)dt\left[\int_{(-I)\cap(-\infty,1/2]}K''\ast\rho_s \varphi dx - \int_{(-I)\cap[1/2,+\infty)}K''\ast\rho_s \varphi dx \right.\\
   & \quad \left.- \int_{I\cap(-\infty,1/2]}K''\ast\rho_s \varphi dx +\int_{I\cap[1/2,+\infty)}K''\ast\rho_s \varphi dx+\int_{(-I)\cup I}K''\ast \rho_s \varphi dx\right].
\end{align*}
Now, since $K''$ and $\rho_s$ are even, the same holds for $K''\ast \rho_s$. Therefore we get
\begin{align}
  & \int_\R |\rho_s -c|\varphi dx +\int_0^T\int_\R\left[|\rho_s-c|\varphi_t -\mathrm{sign}(\rho_s-c)(f(\rho)-f(c))K'\ast\rho_s\varphi_x -f(c)K''\ast \rho_s \varphi\right]\,dxdt\nonumber\\
  & \ \leq 2\int_I \varphi dx - \frac{1}{2}\int_0^{T+1}\xi(t)dt\iint_{I\times I} \left(K''(x-y)+K''(x+y)\right) \varphi(x)dy dx.\label{eq:nonunique1}
\end{align}
Let us now require for simplicity the following additional assumption:
\begin{equation}\label{eq:nonunique_assumption}
  K''(x)>0\qquad \hbox{for all $x\in \R$}.
\end{equation}
Actually, such assumption can be relaxed, see remark \ref{relaxed} below. Then, the last integral in \eqref{eq:nonunique1} is clearly positive, and recalling that $\xi(t)=1$ on $t\in[0,T]$, we can choose $T$ large enough so that the whole right-hand side of \eqref{eq:nonunique1} is strictly negative, thus contradicting Definition \ref{solentropicadef}.

The above argument shows that $\rho_s$ is a weak solution but not an entropy solution. On the other hand, the initial condition $\rho_s$ is $L^\infty$ and $BV$, therefore it must generate an entropy solution according to our main Theorem \ref{main}. Clearly, such solution cannot coincide with $\rho_s$. We have therefore proven the following theorem.

\begin{theorem}\label{thm:nonuniqueness}
  Assume (Av), (AK), and \eqref{eq:nonunique_assumption} are satisfied. Then, there exists an initial condition $\bar\rho \in L^\infty(\R)\cap BV(\R)$ such that the Cauchy problem \eqref{CauchyProblem} has more than one distributional weak solution.
\end{theorem}

\begin{remark}\label{relaxed}
\emph{The assumption \eqref{eq:nonunique_assumption} can be relaxed to include also Gaussian kernels $K(x)=- A e^{-B x^2}$ with $A,B>0$. Indeed, in order to fulfil 
\[\iint_{I\times I} \left(K''(x-y)+K''(x+y)\right) \varphi(x)dy dx>0\]
one has to choose the size of the interval $I$ small enough. We omit the details.}
\end{remark}

\begin{remark}
\emph{The fact that the initial condition $\rho_s$ will not give rise to a stationary solution can also be seen intuitively by using the result in Theorem \ref{main}. Let us approximate $\bar\rho$ with $2(N+1)$ particles with mass $1/(2(N+1))$, with $N$ integer, and with the particles located at $\bar{x}_i$, $i=1,\ldots,2(N+1)$, with
\begin{align*}
  & \bar{x}_i=-1+\frac{i}{2(N+1)}\,,\qquad i=0,\ldots,N\\
  & \bar{x}_i=1/2+\frac{i-N}{2(N+1)}\,,\qquad i=N+1,\ldots,2N+1.
\end{align*}
Let now evolve the particles' positions with the usual ODE system
\[\dot{x}_i=-\frac{v(R_i)}{N}\sum_{j>i}(x_i-x_j) -\frac{v(R_{i-1})}{N}\sum_{j<i}(x_i-x_j).\]
It can be easily proven (we omit the details) that the solution to the particle system preserves the even symmetry of the initial condition.
Moreover, the particle $x_N$ - i.e. the leading particle of the left bump of the initial condition - has a positive initial speed which can be controlled from below by a constant provided that, for example, $K'$ is supported on $\R$ and is strictly monotone on $(0,+\infty)$. Indeed, as all particles $x_i$ with $i<N$ are posed at minimal distance at $t=0$ and the initial distance $x_{N+1}-x_N=1$, we have
\[\dot{x}_N(0)=v(1/N)\frac{1}{N}\sum_{j>N}K'(x_j(0)-x_N(0)) \geq v(1/N)\frac{N+1}{N}K'(2)>v(1/2)K'(2)>0.\]
Similarly, one can show that all particles $i=0,\ldots,N-1$ `move' from their initial position, although their initial speed is zero. A numerical simulation performed in Section \ref{sec:numerics} actually show that for large $N$ the discrete density tends to form a unique bump for large times. Hence, since Theorem \ref{main} shows that the particle solution is arbitrarily close in $L^1_{loc}$ to the entropy solution, this argument supports the evidence that the entropy solution is not stationary. }
\end{remark}

Apart from producing an explicit example of non-uniqueness of weak solutions, the above example shows that there are stationary weak solutions that are not entropy solutions, and therefore cannot be considered as stationary solutions to our problem according to Definition \ref{solentropicadef}. This raises the following natural question: what are the steady states of \eqref{eq:intro_PDE} in the entropy sense? Before asking this question, it will be useful to tackle another task: as the approximating particle system converges to the entropy solutions, detecting the \emph{steady states of \eqref{Odes}} will give us a useful insight about the steady states at the continuum level.

Let us restrict, for simplicity, to the case of an even initial condition $\bar{\rho}$, such that $\|\bar{\rho}\|_{L^1}=1$  and $N\in\N$ fixed. We assume here that $K'$ is supported on the whole $\R$. Consider the following particle configuration,
\begin{equation}\label{stable_conf}
\begin{cases}
\tilde{x}_1=-\frac{1}{2}+\frac{1}{2N},\\
\tilde{x}_{i+1}=\tilde{x}_1+\frac{i}{N},\quad i=1,...,N-2,\\
\tilde{x}_N=\tilde{x}_1+\frac{N-1}{N}=\frac{1}{2}-\frac{1}{2N}.
\end{cases}
\end{equation}
With this choice we get
\[R_i = \frac{1}{N(\tilde{x}_{i+1}-\tilde{x}_i)}=1, \quad v(R_i)=0, \quad \forall i=1,...,N-1,
\]
and it is easy to show that this configuration is a stationary solution for system \eqref{Odes}. Actually, up to space translations, this is the \emph{only} possible stationary solution. In order to prove that, assume that we have a particle configuration as in \eqref{stable_conf} but with only one particle labelled $I$ such that
\[
 \tilde{x}_I=\tilde{x}_{1}+\frac{I-1}{N}, \quad  \tilde{x}_{I+1}=\bar{x}>\tilde{x}_{I}+\frac{1}{N}.
\]
For such a configuration
\[
R_I = \frac{m}{N(\tilde{x}_{I+1}-\tilde{x}_I)}<1, \quad v(R_I)>0,\mbox{ and } \quad v(R_i)=0 \quad\forall i\neq I,
\]
and the $I$ particles evolves according to
\begin{align*}
 &  \dot{\tilde{x}}_I = -\frac{v(R_I)}{N}\sum_{j>I}K'(\tilde{x}_I-\tilde{x}_j)  =-\frac{v(R_I)}{N}\sum_{j>I}K'\left(\frac{1}{N}(I-j)\right)>0,
\end{align*}
and then $\tilde{x}_I$ moves with positive velocity.

We observe that, as $N\to\infty$, the piecewise constant density reconstructed by configuration \eqref{stable_conf} converges in $L^1$ to the step function
\[
\rho_S=\mathbf{1}_{[-\frac{1}{2},\frac{1}{2}]}.
\]
The above discussion suggests that all initial data with multiple bumps only attaining the values $0$ and $1$ are (weak solutions but) not entropy solutions except $\rho_S$. Actually, this statement can be proven exactly in the same way as we proved Theorem \ref{thm:nonuniqueness}, as it is clear that the position of the decreasing discontinuity at $x=-1/2$ and of the increasing discontinuity at $x=1/2$ do not play an essential role. By choosing the test function $\varphi$ suitably, one can easily show that the entropy condition can be contradicted by suitably centring $\varphi$ around the non-admissible discontinuities. We omit the details. As a consequence, we can assert that $\rho_S$ is the \emph{only} stationary solution to \eqref{eq:intro_PDE} in the sense of Definition \ref{solentropicadef}.

\section{Numerical simulations}\label{sec:numerics}
The last section of the paper is devoted to present some numerical experiments based on the particle methods presented in the paper, supporting the results in the previous sections. The qualitative property that emerges more clearly in the simulations below is that solutions tend to aggregate and narrow their support. However, the maximal density constraint avoids the blow-up, and the density profile tends for large times towards the non-trivial stationary pattern presented at the end of the previous section. We compare our particle method with a classical Godunov method for \eqref{eq:intro_continuum}.

\subsection*{Particle simulations}

We first test the particle method introduced in Section \ref{sec:2}. We proceed as follows: we set the number of particles as $N$ and we reconstruct the initial distribution according to \eqref{eq:dscr_IC} (for step functions we simply set the particles initially at distance $\frac{\ell}{N}$ from each other where $\ell$ is the length of the support). Once we have defined the initial distribution, we solve the system \eqref{Odes} with a MATLAB solver and we reconstruct the discrete density as
\begin{equation}\label{eq:central}
 R_i(t)=\frac{m}{2N(x_{i+1}(t)-x_{i-1}(t))}, \quad i=2,N-1.
\end{equation}

The choice of central differences does not effect the particle evolution, since in solving system \eqref{Odes} we define $R_i$ with forward differences. The choice in\eqref{eq:central} is only motivated by the symmetry of the patterns we expect to achieve for large times.
\begin{remark}\label{rem_zero_den}
\emph{In the construction of the discrete densities we get the problem of giving density to the first and the last particles (or only to the last one if we use forward differences). Among all the possible choices we set at zero this two densities, namely
\[
 R_1(t)=R_N(t)=0.
\]
This is a natural choice if we are dealing with step functions but it is not suitable with more general initial conditions, see  \figurename~\ref{fig:cup_IF}.}
\end{remark}

In all the simulations we set
\[
 v(\rho)=1-\rho, \quad K(x)=\frac{C}{\sqrt{2\pi}}e^{-\frac{x^2}{2}} \mbox{ and } N=300.
\]
In the particles evolution we don't fix any time step that is automatically determined by the solver.

The first example we furnish is the case of a single step function with symmetric support,
\begin{equation}\label{eq:sing_step}
 	\bar{\rho}(x) = 0.3 \quad x\in\left[-1,1\right].
\end{equation}
For this initial condition $m=0.6$, so the final configuration will be a step function of value $\rho=1$ supported in $\left[-0.3,0.3\right]$. In  \figurename~\ref{fig:1s_IF} we plot initial (left) and final (right) configurations, while in  \figurename~\ref{fig:1s_Ev} evolution in time is plotted.
\begin{figure}[!ht]
\begin{center}
\begin{minipage}[c]{.45\textwidth}
\includegraphics[width=.95\textwidth]{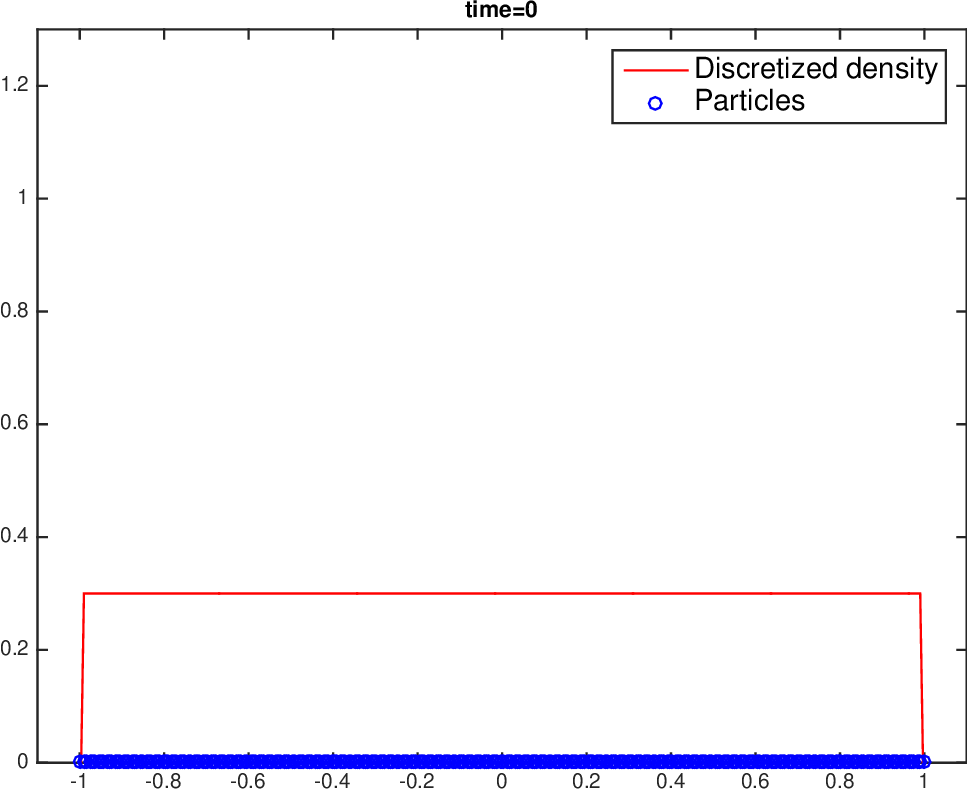}
\end{minipage}
\hspace{3mm}
\begin{minipage}[c]{.45\textwidth}
\includegraphics[width=.95\textwidth]{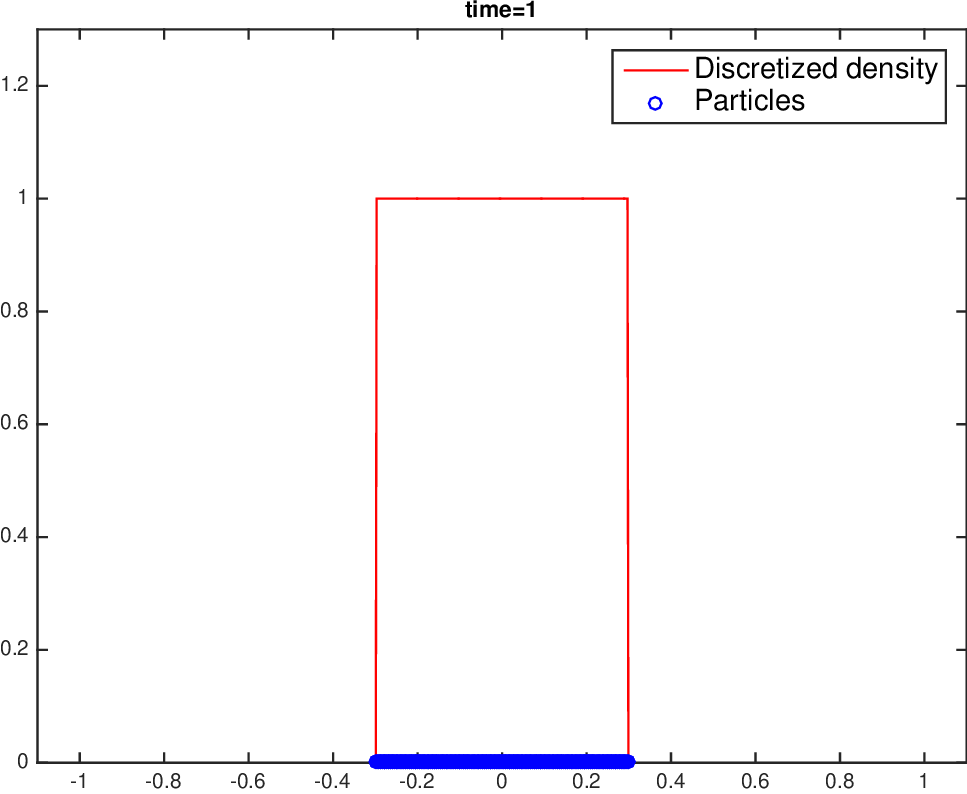}
\end{minipage}
\end{center}
\caption{On the left: initial condition as in \eqref{eq:sing_step}; on the right: the final stationary configuration. We plot the discrete density in (red)-continuous line and the particles positions in (blue)-circles on the bottom of the picture.}
\label{fig:1s_IF}
\end{figure}
\begin{figure}[!ht]
\begin{center}
\begin{minipage}[c]{.70\textwidth}
\includegraphics[width=.95\textwidth]{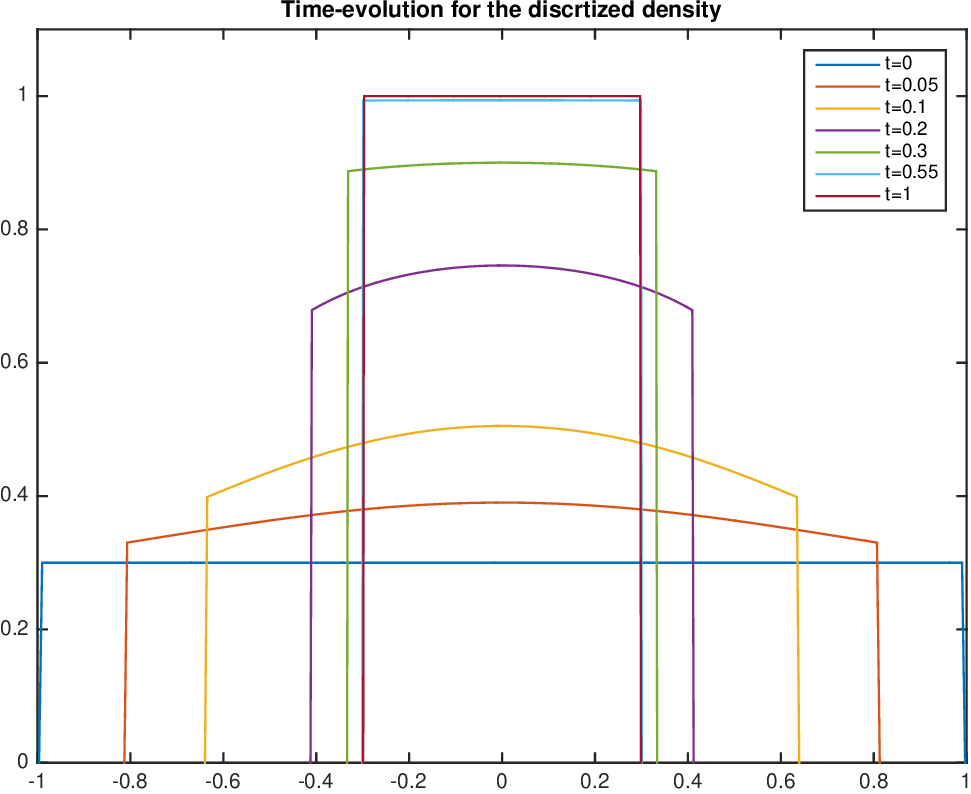}
\end{minipage}
\end{center}
\caption{Evolution of the discrete density for the initial configuration \eqref{eq:sing_step}.}
\label{fig:1s_Ev}
\end{figure}

Next we show the evolution corresponding to the following initial condition,
\begin{equation}\label{eq:id_parabola}
\bar{\rho}(x) = \frac{3}{4}(1-x^2), \quad x\in\left[-1,1\right].
\end{equation}
Even in this case the function is symmetric with respect to the origin so it will converge to the unitary step function supported in $\left[-0.5,0.5\right]$ since $\bar{\rho}$ has normalized mass. As in the previous example initial and final configurations and time evolution of the solution are plotted in \figurename~\ref{fig:cup_IF}.
\begin{figure}[!ht]
\begin{center}
\begin{minipage}[c]{.45\textwidth}
\includegraphics[width=.95\textwidth]{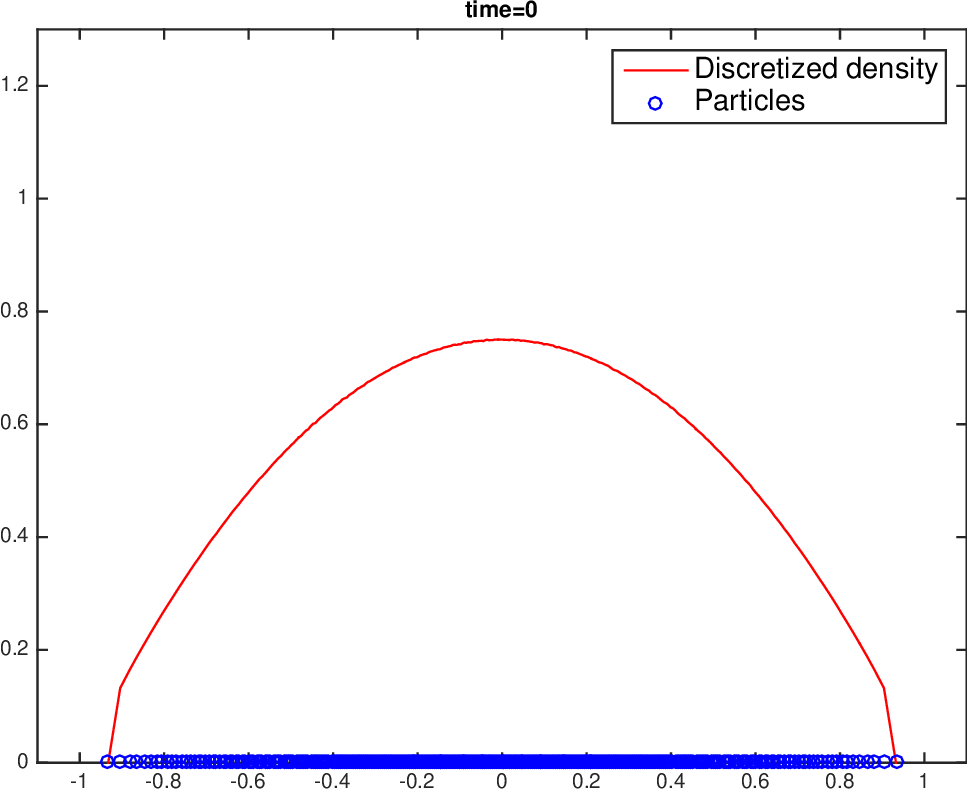}
\end{minipage}
\hspace{3mm}
\begin{minipage}[c]{.45\textwidth}
\includegraphics[width=.95\textwidth]{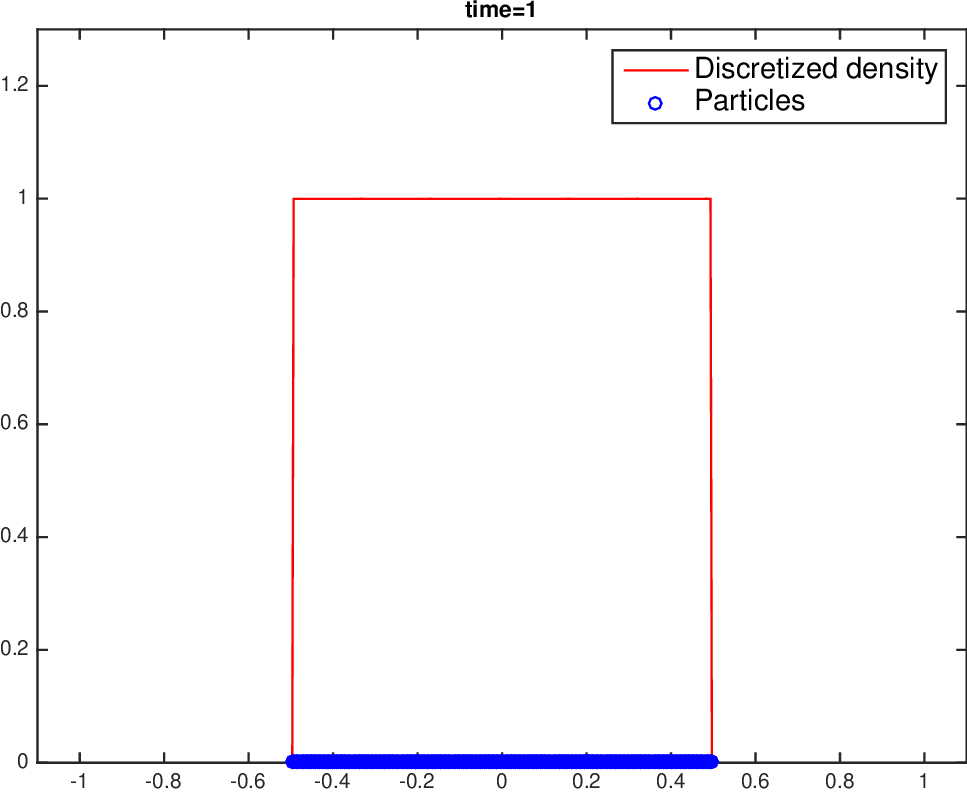}
\end{minipage}
\hspace{3mm}
\begin{minipage}[c]{.70\textwidth}
\vspace{3mm}
\includegraphics[width=.95\textwidth]{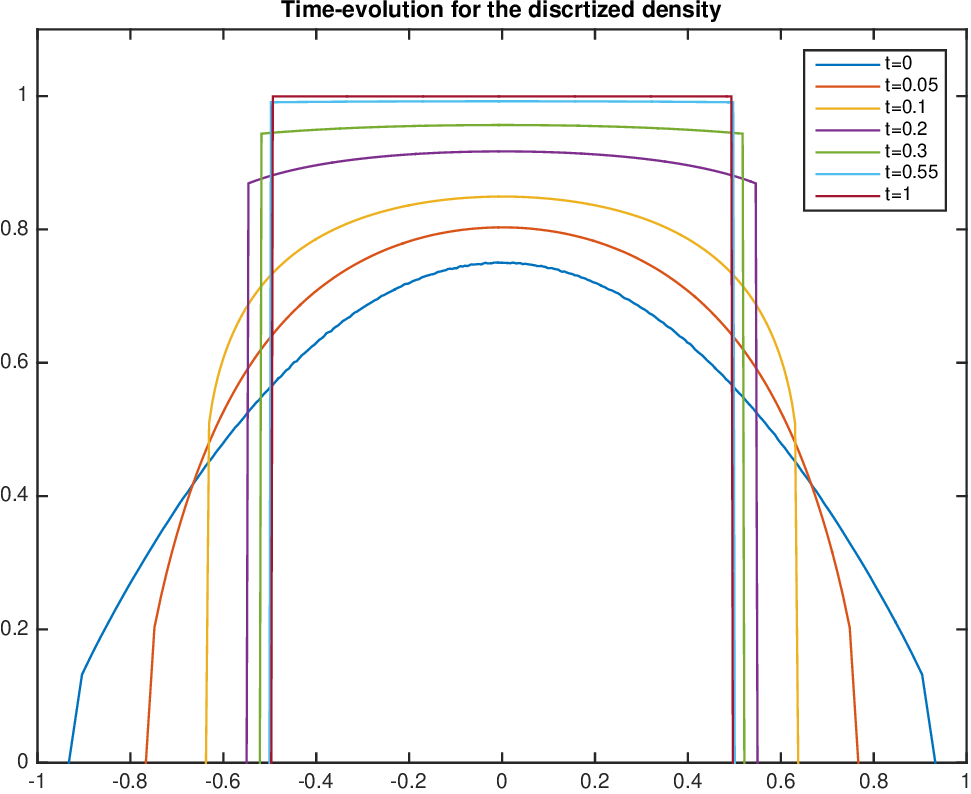}
\end{minipage}
\end{center}
\caption{For the initial condition \eqref{eq:id_parabola} the initial particle configuration is obtained thanks to \eqref{eq:dscr_IC}. The discrete density behaves suitably around all the particles except the first and the last one. See Remark \ref{rem_zero_den}. }
\label{fig:cup_IF}
\end{figure}

We conclude with step functions with disconnected support. We first study the case
\begin{equation}\label{eq:2s_0206}
 	\bar{\rho}(x) = \begin{cases}
  	                            0.2 \quad x\in\left[-0.5,0\right]\\
                                0.6 \quad x\in\left[0.5,1\right]
                   	\end{cases},
\end{equation}
showing that the two bumps merge into a single step. Since symmetry is lost, it is not straightforward to determine where this final configuration will stabilize, but in \figurename~\ref{fig:0206_Ev} we can see that they still aggregate in a step of unitary density and support of length $m$.

\begin{figure}[!ht]
\begin{center}
\begin{minipage}[c]{.45\textwidth}
\includegraphics[width=.95\textwidth]{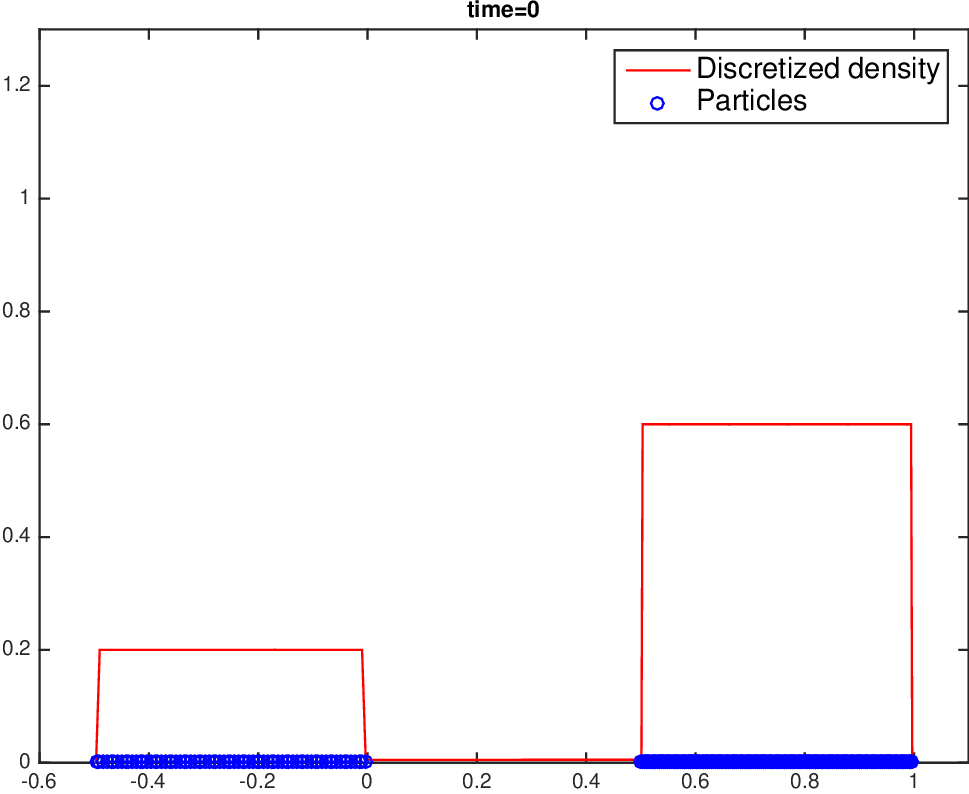}
\end{minipage}
\hspace{3mm}
\begin{minipage}[c]{.45\textwidth}
\includegraphics[width=.95\textwidth]{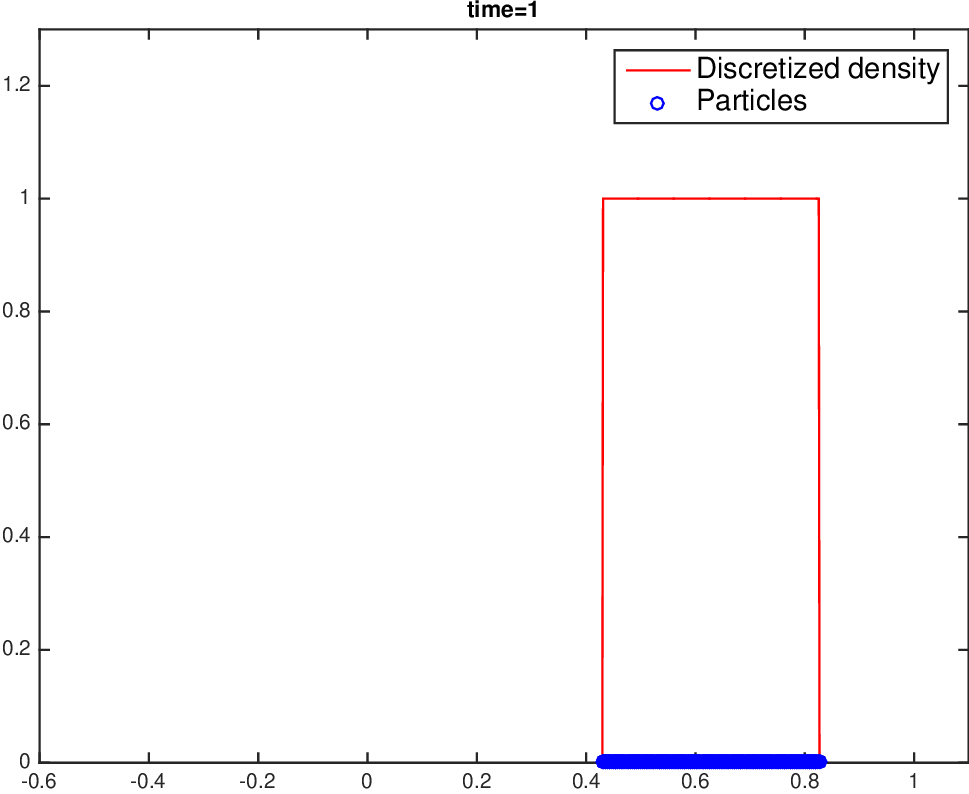}
\end{minipage}
\hspace{3mm}
\begin{minipage}[c]{.70\textwidth}
\vspace{3mm}
\includegraphics[width=.95\textwidth]{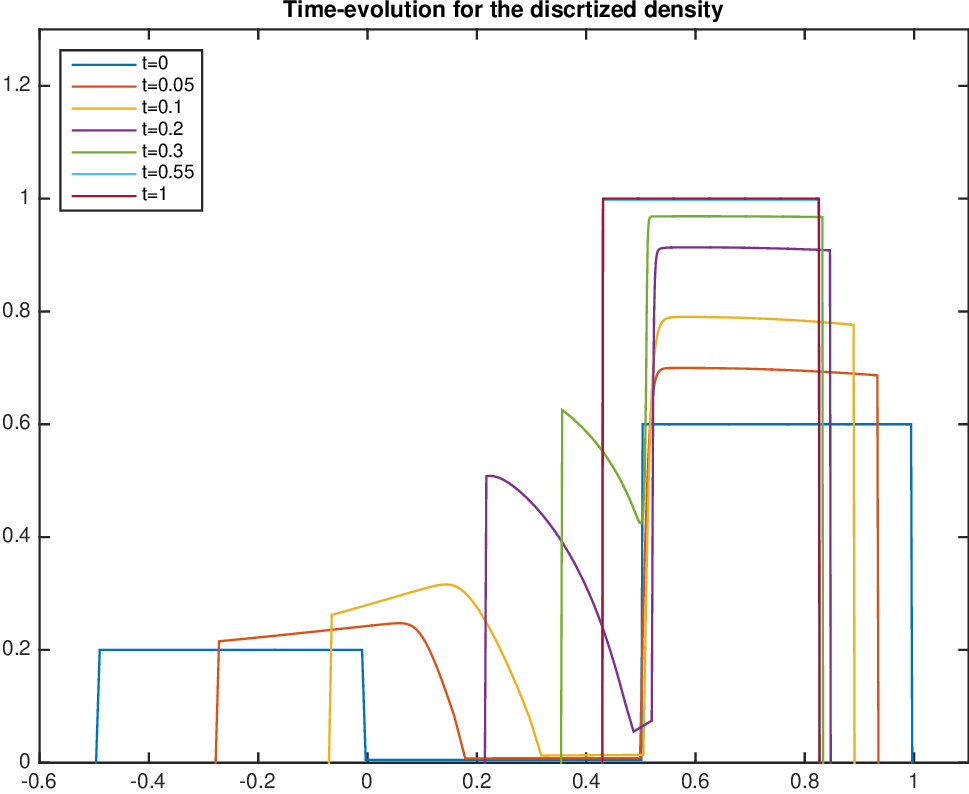}
\end{minipage}
\end{center}
\caption{Evolution of a the two steps initial condition \eqref{eq:2s_0206}. The pattern on the left is the one with less density and moves faster attracted by the one on the right and they merge in a single step of unitary density.}
\label{fig:0206_Ev}
\end{figure}

More interesting is the case of the following initial condition:
\begin{equation}\label{eq:2s_11}
    \bar{\rho}(x) = \begin{cases}
  	                            1 \quad x\in\left[-0.5,0\right]\\
                                1 \quad x\in\left[0.5,1\right]
                   	\end{cases}.
\end{equation}
Note that this is a stationary weak solution to \eqref{eq:intro_continuum} but it is not an entropy solution.
In \figurename~\ref{fig:11_IF} we plot the time evolution of this initial configuration.

\begin{figure}[!ht]
\begin{center}
\begin{minipage}[c]{.45\textwidth}
\includegraphics[width=.95\textwidth]{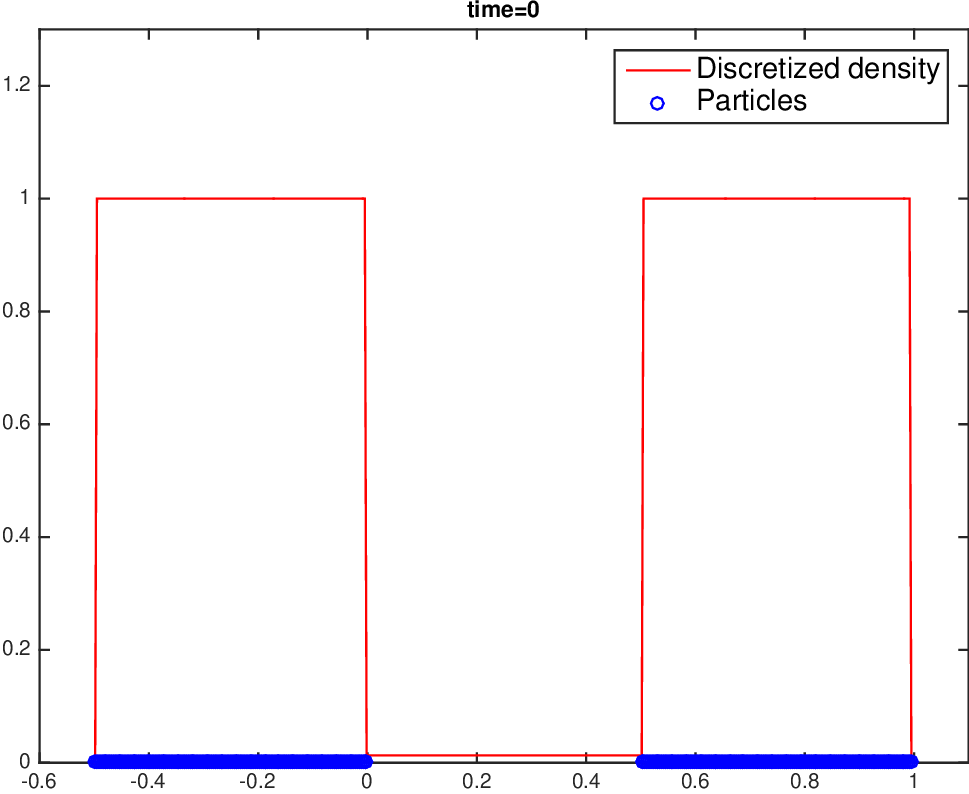}
\end{minipage}
\hspace{3mm}
\begin{minipage}[c]{.45\textwidth}
\includegraphics[width=.95\textwidth]{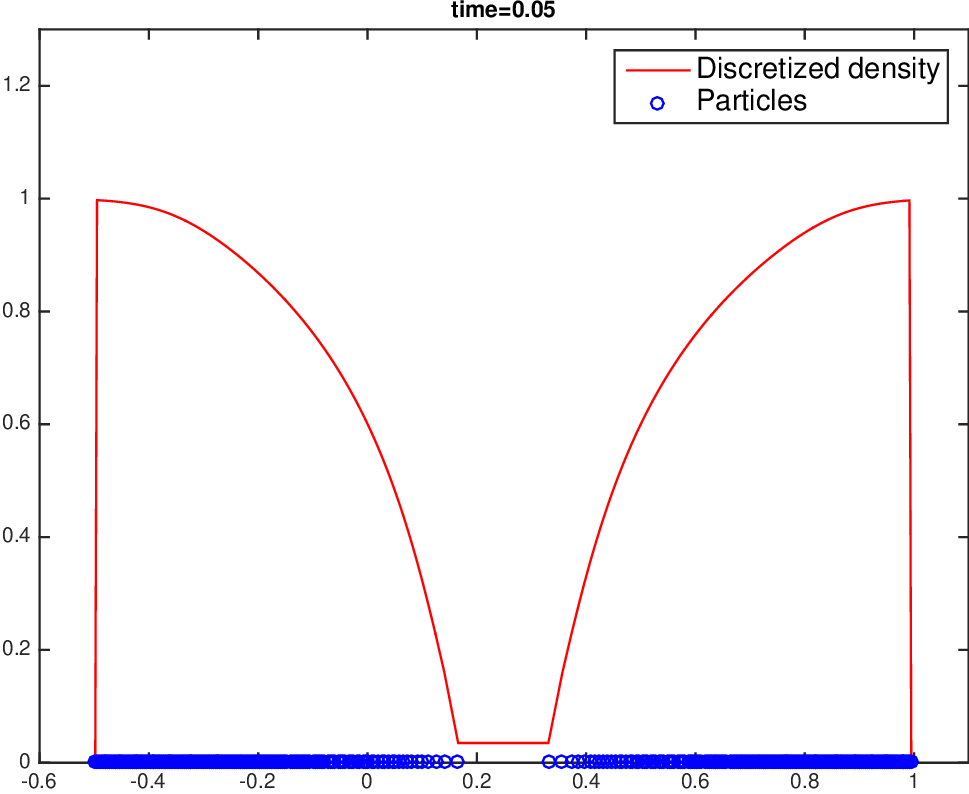}
\end{minipage}
\begin{minipage}[c]{.45\textwidth}
\includegraphics[width=.95\textwidth]{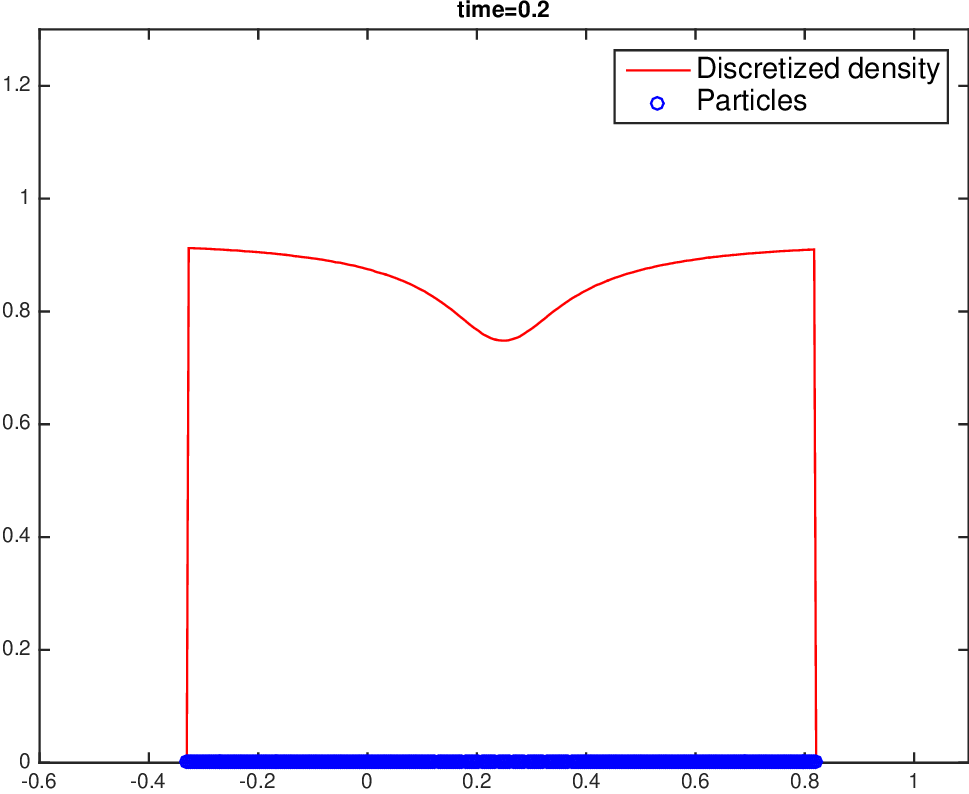}
\end{minipage}
\hspace{3mm}
\begin{minipage}[c]{.45\textwidth}
\includegraphics[width=.95\textwidth]{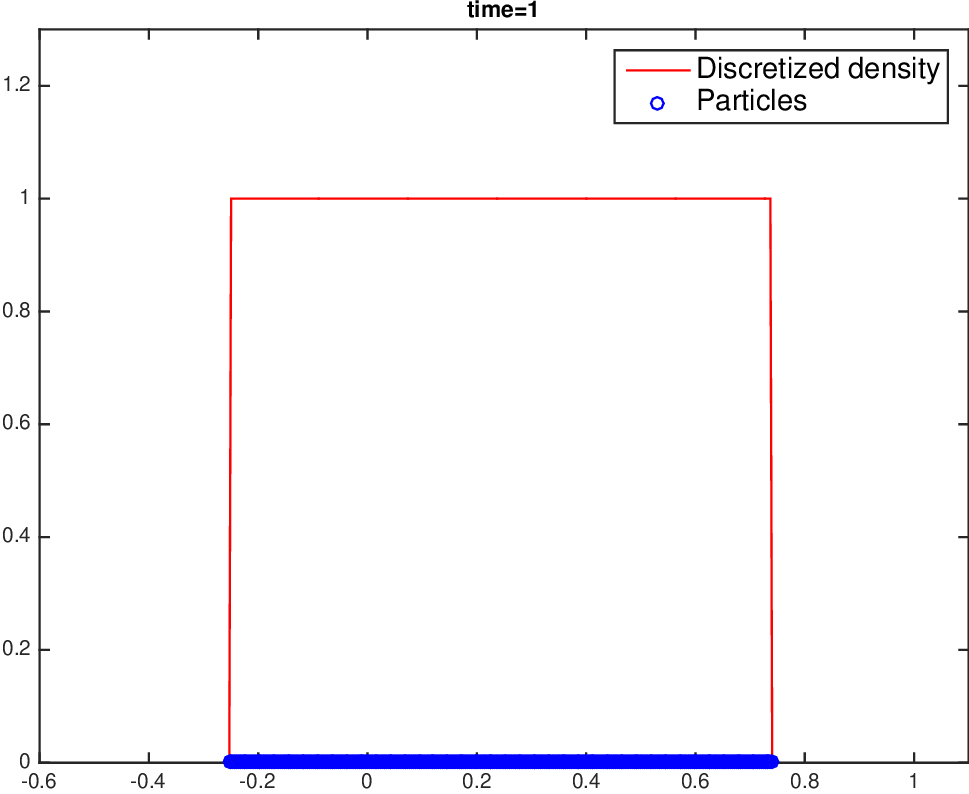}
\end{minipage}
\end{center}
\caption{Solution to \eqref{CauchyProblem} with initial condition \eqref{eq:2s_11}. The initial condition is a weak stationary solution to \eqref{eq:intro_PDE}. However, the particle scheme converges to another solution, actually the unique entropy solution to \eqref{CauchyProblem}. The picture shows how that two `internal' discontinuities are not admissible in the entropy sense, and they are therefore `smoothed' immediately after $t=0$.}
\label{fig:11_IF}
\end{figure}

\subsection*{Comparison with classical Godunov method}
In order to validate the previous simulations we compare the results with a classical Godunov method. The main issue in this case is to dealing with the two directions in the transport term. More precisely, since the kernel $K$ is an even function, we can rephrase \eqref{eq:intro_continuum} as
\begin{equation}\label{eq:god1}
 \partial_t \rho = \partial_x(\rho v(\rho) )K_\rho^{+}(x)+\partial_x(\rho v(\rho)) K_\rho^{-}(x)+\rho v(\rho) K''\ast\rho,
\end{equation}
where
\begin{align*}
 &    K_\rho^{+}(x)=\int_{x\geq y}K'(x-y)\rho(y)dy\geq 0,\\
 &    K_\rho^{-}(x)=\int_{x < y}K'(x-y)\rho(y)dy\leq 0.
 \end{align*}
 The evolution of $\rho$ is driven by two transport fields: $K_\rho^{+}$ pushing the density from left to right $K_\rho^{-}$ pushing the density from right to left. The third term on the r.h.s. in \eqref{eq:god1} plays the role of a source term. Following the standard finite volume approximation procedure on $N$ cells $\left[x_{j-\frac12},x_{j+\frac12}\right]$, the discrete equation reads as
\[
 \frac{d}{dt}\tilde{\rho}_j = K_\rho^{+}(x_j)\frac{F_{j+\frac12}^{+}-F_{j-\frac12}^{+}}{\Delta x} + K_\rho^{-}(x_j)\frac{F_{j+\frac12}^{-}-F_{j-\frac12}^{-}}{\Delta x} + \tilde{\rho}_j v(\tilde{\rho}_j)dK_j
\]
where $F_{j+\frac12}^{+}$ and $F_{j+\frac12}^{-}$ are the Godunov approximations of the fluxes and $dK_j$ is an approximation of the convolution in the reaction term obtained via a quadrature formula. We integrate in time with a time step satisfying the CFL condition of the method.  In \figurename~\ref{fig:Test} we compare the solutions obtained with the two methods in all the examples illustrated above.

 \begin{figure}[!ht]
\begin{center}
\begin{minipage}[c]{.45\textwidth}
\includegraphics[width=.95\textwidth]{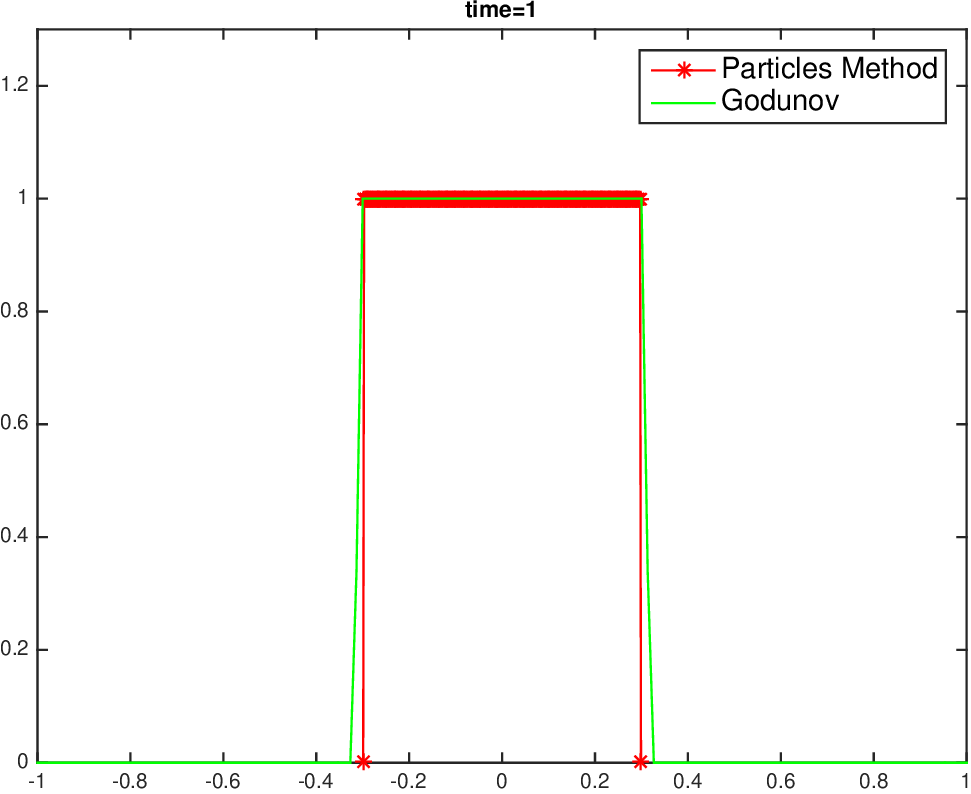}
\end{minipage}
\hspace{3mm}
\begin{minipage}[c]{.45\textwidth}
\includegraphics[width=.95\textwidth]{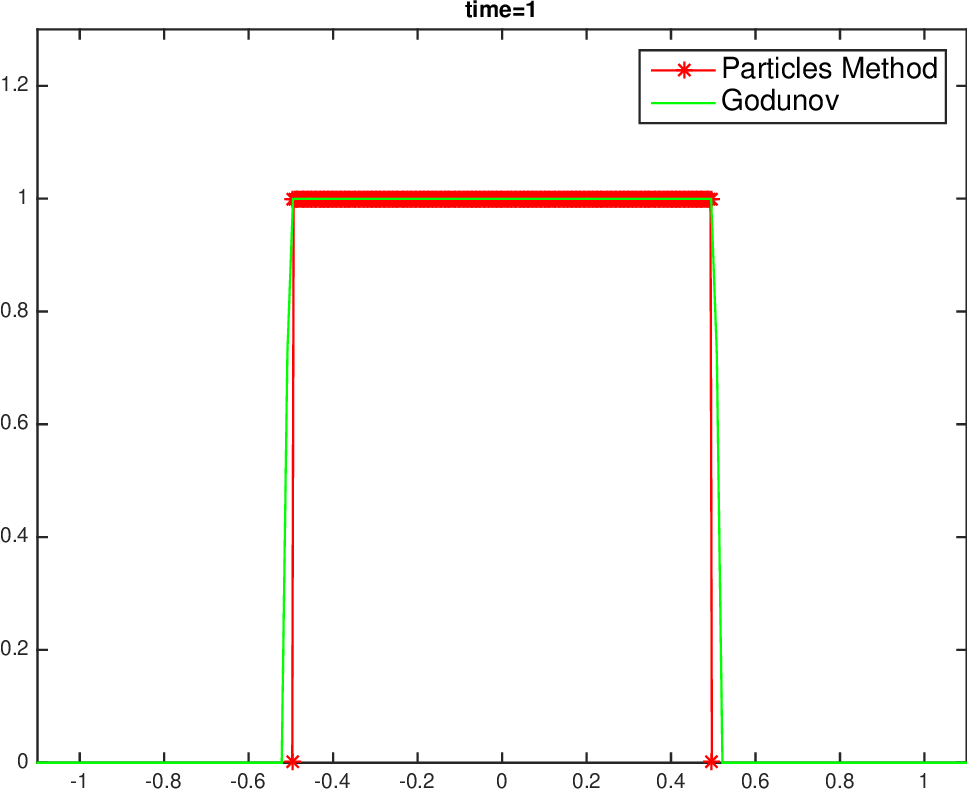}
\end{minipage}
\begin{minipage}[c]{.45\textwidth}
\includegraphics[width=.95\textwidth]{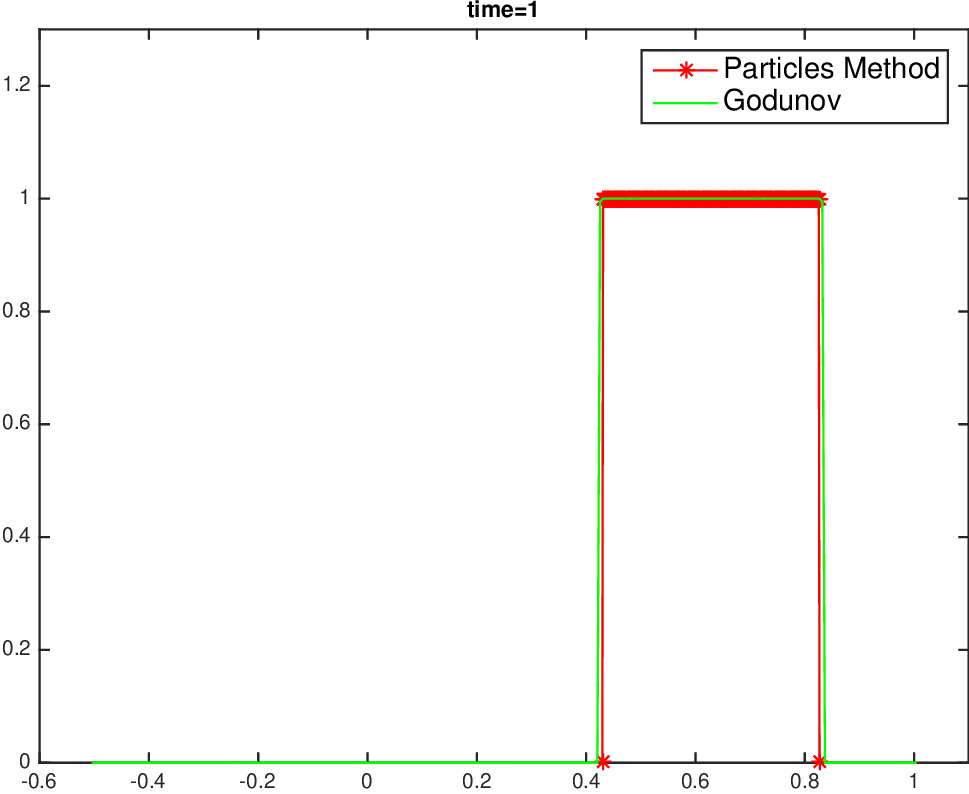}
\end{minipage}
\hspace{3mm}
\begin{minipage}[c]{.45\textwidth}
\includegraphics[width=.95\textwidth]{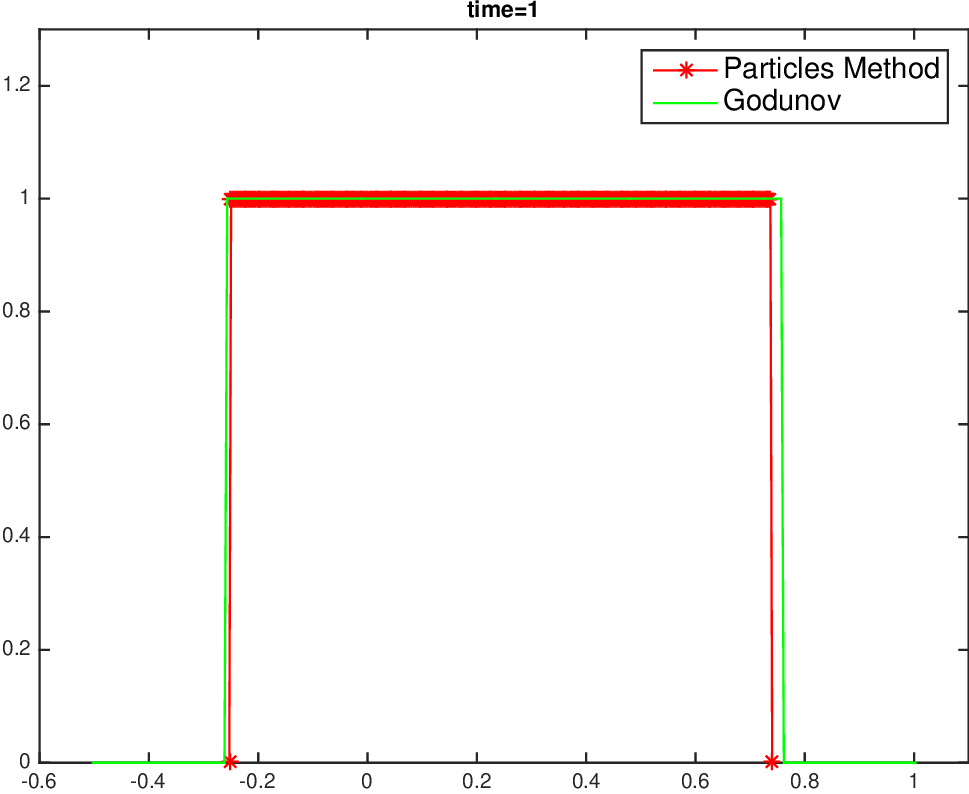}
\end{minipage}
\end{center}
\caption{Comparison between particles (red stars) and Godunov (green continuous line) methods at final time $t=1$. On the top: solutions corresponding to initial condition \eqref{eq:sing_step} (left) and \eqref{eq:id_parabola}(right). On the bottom: final configurations for \eqref{eq:2s_0206} (left) and \eqref{eq:2s_11}.}
\label{fig:Test}
\end{figure}

\section*{Acknowledgements}
The authors acknowledge support from the EU-funded Erasmus Mundus programme `MathMods - Mathematical models in engineering: theory, methods, and applications' at the University of L'Aquila, from the Italian GNAMPA mini-project `Analisi di modelli matematici della fisica,
della biologia e delle scienze sociali', and from the local fund of the University
of L’Aquila `DP-LAND (Deterministic Particles for Local And Nonlocal Dynamics).


\end{document}